\documentclass{amsart}

\usepackage{amsthm}
\usepackage{newlfont}
\usepackage{amsmath}
\usepackage{amssymb}
\usepackage{amsfonts}
\usepackage[mathscr]{euscript}

\usepackage{mathrsfs}

\usepackage{fancybox}
\usepackage{nicefrac}

\usepackage[arrow,curve,matrix,tips]{xy}
\CompileMatrices 

\usepackage{stmaryrd}

\usepackage{ifthen}

\usepackage{hyperref}

\usepackage{enumitem}
\setenumerate[1]{label=(\alph*)}

\usepackage{leftidx}

\hypersetup{
 pdfauthor = {Olaf~M.~Schn{\"u}rer},
 pdftitle = {Perfect Derived Categories of Positively Graded DG Algebras},
 pdfsubject = {Mathematics},
 pdfkeywords = {Differential Graded Module, DG Module, t-Structure,
   Heart, Koszul Duality},
}

\newtheorem{theorem}{Theorem}
\newtheorem{lemma}[theorem]{Lemma}
\newtheorem{corollary}[theorem]{Corollary}

\newtheorem{proposition}[theorem]{Proposition}

\theoremstyle{definition}

\newtheorem{example}[theorem]{Example}
\newtheorem{examples}[theorem]{Examples}

\newtheorem{remark}[theorem]{Remark}

\newcommand{\remarkend}{\leavevmode\unskip\penalty9999 \hbox{}\nobreak\hfill
  \quad\hbox{\ensuremath{\Diamond}}}
\newcommand{\exampleend}{\leavevmode\unskip\penalty9999 \hbox{}\nobreak\hfill
  \quad\hbox{\ensuremath{\Diamond}}}
\newcommand{\examplesend}{\leavevmode\unskip\penalty9999 \hbox{}\nobreak\hfill
  \quad\hbox{\ensuremath{\Diamond}}}

\newcommand{\op}{\operatorname}

\newcommand{\ra}{\rightarrow}

\newcommand{\sra}{\twoheadrightarrow}
\newcommand{\hra}{\hookrightarrow}

\newcommand{\xra}[1]{\xrightarrow{#1}}

\newcommand{\sira}{\xra{\sim}}

\newcommand{\sila}{\overset{\sim}{\leftarrow}}

\newcommand{\ol}[1]{{\overline{#1}}}

\newcommand{\sar}{\ar@{->>}}
\newcommand{\iar}{\ar@{^{(}->}}
\newcommand{\gar}{\ar@{=}}

\newcommand{\Bl}[1]{{\mathbb{#1}}}
\newcommand{\DZ}{\Bl{Z}}
\newcommand{\DN}{\Bl{N}}

\newcommand{\DR}{\Bl{R}}
\newcommand{\DC}{\Bl{C}}

\newcommand{\DK}{{\Bl{K}}}

 \newcommand{\Hom}{\op{Hom}}
 \newcommand{\complexHom}{{\mkern3mu\mathcal{H}{\op{om}}\mkern3mu}}
 
\newcommand{\complexEnd}{{\mkern3mu\mathcal{E}{\op{nd}}\mkern3mu}}

\newcommand{\Ext}{\op{Ext}}

\newcommand{\Lotimes}{\overset{L}{\otimes}}

\newcommand{\Mod}{\op{Mod}}

\newcommand{\gmod}{\op{gmod}}

\newcommand{\gMod}{\op{gMod}}
\newcommand{\Mof}{\op{mod}}
\newcommand{\gMof}{\op{gmod}}
\newcommand{\gProjf}{\op{gproj}}

\newcommand{\Modover}[1]{{{#1}\text{-}\Mod}}
\newcommand{\Mofover}[1]{{{#1}\text{-}\Mof}}

\newcommand{\soc}{\op{soc}}

\newcommand{\End}{\op{End}}

\newcommand{\Mat}{\op{Mat}}

\newcommand{\length}{\lambda}

\newcommand{\leftadjointtores}{\op{prod}}
\newcommand{\pro}{\leftadjointtores}

\newcommand{\Gr}{\op{Gr}}

\newcommand{\Kern}{\op{ker}}

\newcommand{\Kokern}{\op{cok}}
\newcommand{\Bild}{\op{im}}

\newcommand{\bild}{\op{im}}

\newcommand{\supp}{\op{supp}}

\newcommand{\mar}{{\ar@{|->}}}

\newcommand{\dgDer}{\op{dgDer}}
\newcommand{\dgHot}{\op{dgHot}}
\newcommand{\dgHotproj}{{\op{dgHotp}}}
\newcommand{\dgMod}{\op{dgMod}}

\newcommand{\dgPer}{\op{dg\textbf{P}er}}
\newcommand{\satzdgPer}{\op{dg\emph{\textbf{P}}er}}
\newcommand{\dgPerDer}{{\op{dgPer}}}
\newcommand{\dgPrae}{\op{dg\textbf{P}rae}}
\newcommand{\satzdgPrae}{\op{dg\emph{\textbf{P}}rae}}
\newcommand{\dgPraeDer}{\op{dgPrae}}

\newcommand{\dgFilt}{\op{dg\textbf{F}ilt}}
\newcommand{\satzdgFilt}{\op{dg\emph{\textbf{F}}ilt}}
\newcommand{\dgFilMod}{\op{dgFlag}}

\newcommand{\dgFiltDer}{\op{dgFilt}}

\newcommand{\heart}{\heartsuit}

\newcommand{\tzmat}[4]{{\big[\begin{smallmatrix} {#1} & {#2} \\ {#3} & {#4} \end{smallmatrix}\big]}}

\newcommand{\svek}[2]{{\big[\begin{smallmatrix} {#1} \\ {#2} \end{smallmatrix}\big]}}

\renewcommand{\tilde}[1]{\widetilde{#1}}

\renewcommand{\hat}[1]{\widehat{#1}}

\newcommand{\comp}{\circ}

\newcommand{\opp}{{\op{op}}}

\newcommand{\define}[1]{{\textbf{#1}}}
\newcommand{\stress}[1]{{\textit{#1}}}

\begin{document}

\title[Perfect Derived Categories]{Perfect Derived Categories of\\
  Positively Graded DG Algebras}
\author{Olaf M.\ Schn{\"u}rer}
\address{Mathematisches Institut, Universit{\"a}t Bonn, 
Beringstra\ss{}e 1, D-53115 Bonn, Germany}
\email{olaf.schnuerer@math.uni-bonn.de}
\thanks{Supported by a grant of the state of Baden-W{\"u}rttemberg}

\date{October 2008, revised January 2010.}



\begin{abstract}
  We investigate the perfect derived category $\dgPerDer(\mathcal{A})$ of a
  positively graded differential graded (dg) algebra $\mathcal{A}$ whose degree zero
  part is a dg 
  subalgebra and semisimple as a ring. 
  We introduce an equivalent subcategory of $\dgPerDer(\mathcal{A})$
  whose objects are easy to describe, define a t-structure on
  $\dgPerDer(\mathcal{A})$ and study its heart. We show that
  $\dgPerDer(\mathcal{A})$ is a Krull-Remak-Schmidt category.  
  Then we consider the heart in the case that $\mathcal{A}$ is a
  Koszul ring with differential zero satisfying some finiteness
  conditions. 
\end{abstract}

\maketitle

\setcounter{tocdepth}{1}
\tableofcontents

\section{Introduction}
\label{cha:introduction}

Let $k$ be a commutative ring and 
$\mathcal{A}=(A=\bigoplus_{i \in \DZ} A^i, d)$ 
a differential graded $k$-algebra (= dg algebra). 
Let $\dgDer(\mathcal{A})$ be the derived category of dg (right)
modules over $\mathcal{A}$ (= $\mathcal{A}$-modules), 
and $\dgPerDer(\mathcal{A})$ the perfect derived category,
i.\,e.\ the smallest strict full triangulated subcategory of
$\dgDer(\mathcal{A})$ containing $\mathcal{A}$ and closed under
forming direct summands; the objects of $\dgPerDer(\mathcal{A})$ are
precisely the compact objects of $\dgDer(\mathcal{A})$ (see
\cite{Keller-deriving-dg-cat, Keller-construction-of-triangle-equiv}). 

The aim of this article is to provide some description of
$\dgPerDer(\mathcal{A})$ and to define a t-structure on
this category if $\mathcal{A}=(A,d)$ is a dg algebra
satisfying the following conditions:
\begin{enumerate}[label={(P\arabic*)}]
\item 
\label{enum:intro-pg}
$A$ is positively graded, i.\,e.\ $A^i=0$ for $i < 0$;
\item 
\label{enum:intro-ss}
$A^0$ is a semisimple ring;
\item 
\label{enum:intro-sdga}
the differential of $\mathcal{A}$ vanishes on $A^0$, i.\,e.\ $d(A^0)=0$. 
\end{enumerate}
At the end of this introduction we 
explain our main motivation 
for studying such perfect derived categories.
The only related and in fact motivating description (with a definition
of a t-structure) we know of can be found
in \cite[11]{BL}; the dg algebra $(\DR[X_1, \dots, X_n], d=0)$
considered there is a polynomial algebra with generators in strictly
positive even degrees.

We give an account of the results of this article, always assuming
that $\mathcal{A}$ is a dg algebra satisfying the 
conditions \ref{enum:pg}-\ref{enum:sdga}. 

\subsection*{Alternative descriptions of the perfect derived category}
\label{sec:altern-descr-dgperd}

The semisimple ring $A^0$ has only a finite number
of non-isomorphic simple (right) modules $(L_x)_{x \in W}$.
We view $A^0$ as a dg subalgebra $\mathcal{A}^0$ of $\mathcal{A}$ and the
$L_x$ as $\mathcal{A}^0$-modules concentrated in degree zero.
Let $\dgPraeDer(\mathcal{A})$ be the smallest strict full triangulated
subcategory of $\dgDer(\mathcal{A})$ that contains all
$\mathcal{A}$-modules $\hat L_x:= L_x \otimes_{\mathcal{A}^0}
\mathcal{A}$. (The name $\dgPraeDer$ was chosen because this category seemed
to be a precursor of the perfect derived category: It is not required
to be closed under taking direct summands. But indeed it is closed
under this operation, cf.\ Theorem~\ref{t:describe}.)
Define $\dgFiltDer(\mathcal{A})$ to be the full subcategory of
$\dgDer(\mathcal{A})$ whose objects are $\mathcal{A}$-modules
$M$ admitting a finite filtration 
$0 = F_0(M) \subset F_1(M) \subset \dots \subset
F_n(M) =M$ by dg submodules with subquotients
$F_{i}(M)/F_{i-1}(M) \cong \{l_i\}\hat L_{x_i}$
for suitable $l_1 \geq l_2 \geq \dots \geq l_n$ and $x_i \in W$; here
$\{1\}$ denotes the shift functor.
We have inclusions $\dgFiltDer(\mathcal{A}) \subset
\dgPraeDer(\mathcal{A}) \subset \dgPerDer(\mathcal{A})$.

\begin{theorem}[{cf.\ Theorems \ref{t:filt-iso-prae} and \ref{t:t-cat-prae}}]
  \label{t:describe}
  \rule{0mm}{1mm}
  \begin{enumerate}[label={(\arabic*)}]
  \item 
    $\dgPraeDer(\mathcal{A})$ is closed under taking direct summands,
    i.\,e.\ 
    \begin{equation*}
      \dgPraeDer(\mathcal{A})=\dgPerDer(\mathcal{A}).
    \end{equation*}
  \item The inclusion
    $\dgFiltDer(\mathcal{A})\subset \dgPerDer(\mathcal{A})$ is an
    equivalence of categories.
  \end{enumerate}
\end{theorem}

The proof of this theorem relies on the existence of a bounded 
t-structure (as
described below) on $\dgPraeDer(\mathcal{A})$.

The objects of $\dgFiltDer(\mathcal{A})$ can be characterized as the
homotopically projective objects in $\dgPerDer(\mathcal{A})$ that are
homotopically minimal, cf.\ Proposition \ref{p:dgper-tfae}.
The equivalence $\dgFiltDer(\mathcal{A}) \subset
\dgPerDer(\mathcal{A})$ enables us to prove that $\dgPerDer(\mathcal{A})$
is a Krull-Remak-Schmidt category, cf.\ Proposition
\ref{p:krs-category}.

\subsection*{t-structure}
\label{sec:t-structure-2}
Let $\dgPerDer^{\leq 0}$ (and $\dgPerDer^{\geq 0}$) be the full
subcategories of $\dgPerDer(\mathcal{A})$ consisting of objects
$\mathcal{M}$ such that $H^i(\mathcal{M} \Lotimes_{\mathcal{A}} \mathcal{A}^0)$
vanishes for $i > 0$ (for $i < 0$, respectively);
here 
$(? \Lotimes_{\mathcal{A}}\mathcal{A}^0)$ is the left derived functor of
the extension of scalars functor 
$(? \otimes_{\mathcal{A}}\mathcal{A}^0)$.
Let $\dgMod(\mathcal{A})$ be the abelian category of
$\mathcal{A}$-modules
and $\dgFilMod(\mathcal{A})$ the full subcategory
consisting of objects that have an 
$\hat{L}_x$-flag, i.\,e.\ a
finite filtration with subquotients isomorphic to objects of $\{\hat{L}_x\}_{x \in W}$
(without shifts). 

\begin{theorem}[{cf.\ Theorem \ref{t:t-cat-prae} and Propositions
    \ref{p:dgfilmod-abelian-sub}, \ref{p:dgfilmod-equi}}]
  \label{t:t-cat-per}
  \rule{0mm}{1mm}
  \begin{enumerate}[label={(\arabic*)}]
  \item 
    $(\dgPerDer^{\leq 0}, \dgPerDer^{\geq 0})$ defines a bounded
    (hence non-de\-ge\-ner\-ate) 
    t-structure on
    $\dgPerDer(\mathcal{A})$.
  \item 
  The heart $\heart$ of this t-structure is
  equivalent to $\dgFilMod(\mathcal{A})$. More precisely,
  $\dgFilMod(\mathcal{A})$ is a full abelian subcategory
  of $\dgMod(\mathcal{A})$ and the obvious functor
  $\dgMod(\mathcal{A}) \ra \dgDer(\mathcal{A})$ 
  induces an equivalence 
  $\dgFilMod(\mathcal{A}) \sira \heart$.
  \item 
    Any object in $\heart$ has finite length, and  
    the simple objects in $\heart$ are (up to isomorphism) the
    $\{\hat{L}_x\}_{x \in W}$.
  \end{enumerate}
\end{theorem}

The truncation functors of this t-structure 
have a very simple description on $\dgFiltDer(\mathcal{A})$.

\subsection*{Koszul case}
\label{sec:koszul-case}

Assume now that the differential of $\mathcal{A}$ vanishes and that
the underlying graded ring $A$ is a Koszul ring (cf.\ \cite{BGS}). 
Let
$E(A):=\Ext^\bullet_A(A^0,A^0)$ be the graded ring
of self-extensions of the right
$A$-module $A^0=A/A^+$ (in the category of (non-graded) $A$-modules). 

\begin{theorem}[{cf.\ Theorem \ref{t:endo-koszul}}]
  Assume that $A$ is a Koszul ring with a Koszul resolution of finite
  length with finitely generated components. 
  Let $\mathcal{A}=(A,d=0)$.
  Then the heart $\heart$
  of $\dgPerDer(\mathcal{A})$ is
  equivalent to the opposite category of the category of finitely
  generated left $E(A)$-modules.
\end{theorem}

\subsection*{Motivation}
\label{sec:motivation}

We became interested in perfect derived categories when we studied the
Borel-equivariant 
derived category of sheaves on the flag variety of a 
complex reductive group. 
We show in \cite{OSdiss} 
(or \cite{OSdiss-equi-mathz}) that this category
with its perverse t-structure is t-equivalent to the perfect derived category
$\dgPerDer(\mathcal{E})$ for some dg algebra $\mathcal{E}$
(that meets the conditions \ref{enum:pg}-\ref{enum:sdga}).
More precisely, $\mathcal{E}$ is the graded
algebra of self-extensions of the direct sum of the simple equivariant
perverse sheaves and has differential $d=0$. 
For the proof of this equivalence we need the t-structures 
$(\dgPerDer^{\leq 0}, \dgPerDer^{\geq 0})$ introduced above on several
categories of the form $\dgPerDer(\mathcal{A})$.

There are similar equivalences between equivariant derived categories
of sheaves and categories of the form
$\dgPraeDer(\mathcal{E})=\dgPerDer(\mathcal{E})$, see
\cite{BL, Luntstoric, guillermou}.
The strategy to obtain these equivalences
is quite general (see \cite[0.3]{Luntstoric}), the tricky point
however is to establish the
formality of some dg algebra $\mathcal{B}$ whose perfect derived
category is equivalent to the considered category of
sheaves. Formality means that $\mathcal{B}$ and
its cohomology $H(\mathcal{B})$ (a dg algebra with differential zero)
can be connected by a sequence of quasi-isomorphisms of dg algebras. 
Since each quasi-isomorphism of dg algebras induces an equivalence
between their (perfect) derived categories, formality enables us to
consider the more accessible dg algebra $H(\mathcal{B})$ instead of
$\mathcal{B}$.
Moreover, usually $H(\mathcal{B})$ identifies with the extension
algebra of some nice object from the geometric side.

This general strategy shows that categories of the form
$\dgPraeDer(\mathcal{A})=\dgPerDer(\mathcal{A})$ are natural candidates for
describing certain triangulated categories.

\subsection*{Overview}
\label{sec:overview-1}
This article is organized as follows.
In Chapters
\ref{sec:preliminaries} and \ref{sec:review-dg-modules}
we introduce our notation, prove some
basic results on graded modules and recall some results on dg modules. 
We show the equivalence $\dgFiltDer(\mathcal{A}) \subset
\dgPraeDer(\mathcal{A})$ in Chapter \ref{sec:filtered-dg-modules}.
In the following two chapters we define the t-structure on
$\dgPerDer(\mathcal{A})$, show that
$\dgPraeDer(\mathcal{A})=\dgPerDer(\mathcal{A})$ and describe the heart.
We give alternative characterizations of the objects of
$\dgFiltDer(\mathcal{A})$ in Chapter \ref{sec:homot-minim-dg}.
Chapter \ref{sec:indecomposables} contains some results on
indecomposables, a Fitting lemma for objects of
$\dgFiltDer(\mathcal{A})$ and the proof 
that 
$\dgPerDer(\mathcal{A})$ is a
Krull-Remak-Schmidt category. The last chapter (which is independent of
Chapter \ref{sec:indecomposables}) concerns the case that
$A$ is a Koszul ring.

\subsection*{Acknowledgments}
\label{sec:acknowledgements}

Most of the results of this article (apart from some improvements and Chapters
\ref{sec:homot-minim-dg} and \ref{sec:indecomposables}) can be
found in the algebraic part 
of my thesis \cite{OSdiss}. I am grateful to my advisor Wolfgang Soergel who
introduced me to the dg world. 
I would like to thank 
Peter Fiebig, 
Bernhard Keller, 
Henning Krau\-se,
Catharina Stroppel and
Geordie Williamson
for helpful discussions and interest.
I wish to thank the referee for carefully reading the manuscript and
for questions and comments.

\section{Preliminaries}
\label{sec:preliminaries}

We introduce our notation and prove some easy and probably well-known statements
that are crucial for the rest of this article.
We fix some commutative ring $k$.
All rings and algebras (= $k$-algebras) are assumed to be associative and unital.

If $A$ is a ring, we denote the category of
right $A$-modules by $\Mod(A)$ and the full subcategory of finitely
generated modules by $\Mof(A)$.
If $A$ is graded (= $\DZ$-graded) we write $\gMod(A)$ for the category
of graded modules and $\gMof(A)$ for the full subcategory of finitely
generated graded modules.

Let $A=\bigoplus_{i \in \DN} A^i$ be a
positively graded $k$-algebra.
Let $A^+= \bigoplus_{i > 0} A^i$.
The projection $A \ra A/A^+=A^0$ gives rise to the extension of scalars functor 
\begin{equation*}
\gMod(A) \ra\gMod(A^0), \quad M \mapsto \ol{M}:=M/MA^+=M\otimes_A A^0  .
\end{equation*}

\begin{lemma}
  \label{l:isocrit}
  Let $f:M \ra N$ be a morphism in $\gMof(A)$.
  If $\ol{f}:\ol M \ra \ol N$ is an isomorphism and $N$ is flat as an
  $A$-module, then $f$ is an isomorphism.
\end{lemma}
\begin{proof}
  We use the following trivial observation:
  If a graded $A$-module $X$ with $X^i=0$ for $i \ll 0$ satisfies
  $\ol X=X/XA^+=0$, then $X=0$.

  Since $(? \otimes_A A^0)$ is right exact we obtain
  $\ol{\Kokern f}=\Kokern \ol f=0$, and our observation implies
  $\Kokern f=0$. 
  Applying $(? \otimes_A A^0)$ to 
  the short exact sequence $(\Kern f, M,
  N)$ and using the
  flatness of $N$, we see that $0=\ol{\Kern f}$. Our observation shows
  that $\Kern f =0$. So $f$ is an isomorphism.
\end{proof}

We also have the extension of scalars functor coming from the
inclusion $A^0 \subset A$,
\begin{equation}
  \label{eq:prodhat}
  \pro_{A^0}^A:  
  \gMod(A^0) \ra \gMod(A), \quad M \mapsto \hat M := M \otimes_{A^0} A.
\end{equation}
We often view $A^0$-modules as graded $A^0$-modules concentrated in
degree zero.

Assume now that $A^0$ is a semisimple ring.
Then $A^0$ has
only a finite number
of non-isomorphic simple (right) modules $(L_x)_{x \in W}$.
In particular we obtain projective graded $A$-modules $\hat L_x= L_x
\otimes_{A^0} A$. 

\begin{lemma}
  \label{l:zero-or-iso}
    Let $x$, $y \in W$.
    If $f: \hat L_x \ra \hat L_y$ is a non-zero morphism in
    $\gMod(A)$, it is an isomorphism and $x=y$.
\end{lemma}

\begin{proof}
  Since $\hat{L}_x$ and $\hat{L}_y$ are generated in degree zero (=
  generated as an $A$-module by their degree zero components),
  $\ol{f}: L_x \ra L_y$ is non-zero, hence an isomorphism and $x=y$.
  Now use Lemma \ref{l:isocrit}.
\end{proof}

Let $\gProjf(A)$ be the full subcategory of projective objects in
$\gMof(A)$.
It is clear that any finite direct sum of shifted objects $\hat L_x$ is in
$\gProjf(A)$. 
The converse is also true. We include the proof for completeness.

\begin{lemma}
  \label{l:eegproj}
  Each object of
  $\gProjf(A)$ is isomorphic to a finite direct sum of shifted $\hat
  L_x$, for $x \in W$. 
\end{lemma}

\begin{proof}
  Let $P$ be in $\gProjf(A)$. We view the canonical morphism $P \ra
  \ol{P}$ as a morphism of $A^0$-modules. Since $A^0$ is semisimple,
  there is a splitting $\sigma: \ol P \ra P$. Since $\pro_{A^0}^A$ is
  left adjoint to the restriction functor $\gMod(A) \ra \gMod(A^0)$
  coming from the inclusion $A^0 \subset A$, we get a morphism
  $\hat\sigma: \ol{P} \otimes_{A^0} A \ra P$ in $\gMof(A)$.
  Since $P$ is $A$-flat, we deduce from Lemma \ref{l:isocrit} that
  $\hat\sigma$ is an isomorphism.
  But $\ol P \otimes_{A^0} A$ has the required form.
\end{proof}

\section{Differential Graded Modules}
\label{sec:review-dg-modules}

We review the language of differential graded
(dg) modules over a dg algebra (see \cite{Keller-deriving-dg-cat,
  Keller-construction-of-triangle-equiv, BL}).

Let 
$\mathcal{A}=(A=\bigoplus_{i \in \DZ} A^i, d)$ 
be a differential graded $k$-algebra (= dg algebra). 
A dg (right) module over $\mathcal{A}$ will also be called an
$\mathcal{A}$-module or a dg module if there is no doubt about the dg algebra.  
We often write $M$ for a dg module $(M, d_M)$.
We consider the category $\dgMod(\mathcal{A})$ of dg modules, the
homotopy category $\dgHot(\mathcal{A})$ and the derived category
$\dgDer(\mathcal{A})$ of dg modules.  
(In
\cite{Keller-deriving-dg-cat}, these categories are denoted by
$\mathcal{C}(\mathcal{A})$, $\mathcal{H}(\mathcal{A})$ and
$\mathcal{D}(\mathcal{A})$ respectively.) We often omit $\mathcal{A}$
from the notation.
We denote the shift functor on all these categories (and on
$\gMod(A)$ and $\gMof(A)$) by $M \mapsto \{1\}M$, e.\,g.\
$(\{1\}M)^i=M^{i+1}$, $d_{\{1\}M}= -d_M$. 
We define $\{n\}=\{1\}^n$ for $n
\in \DZ$.

The homotopy category $\dgHot(\mathcal{A})$ with the shift
functor $\{1\}$ and the distinguished triangles isomorphic to standard
triangles (= mapping cones of morphisms) is a triangulated category.
Any short exact sequence $0 \ra K \ra M \ra
N \ra 0$ of $\mathcal{A}$-modules that is $A$-split (= it
splits in $\gMod(A)$) can be completed to a
distinguished triangle 
${K} \ra {M} \ra {N} \ra \{1\}{K}$ in
$\dgHot(\mathcal{A})$. 
The category $\dgDer(\mathcal{A})$ inherits the triangulation
from $\dgHot(\mathcal{A})$.
Since $\dgDer(\mathcal{A})$ has infinite direct sums, every idempotent
in $\dgDer(\mathcal{A})$ splits 
(see \cite[Prop.~3.2]{boekstedt-neeman-homotopy}).

A dg module ${P}$ is called \define{homotopically projective}
(\cite{Keller-construction-of-triangle-equiv}),
if it satisfies one of
the following equivalent conditions
(\cite[10.12.2.2]{BL}):
\begin{enumerate}
\item $\Hom_{\dgHot}({P}, ?) = \Hom_{\dgDer}({P},?)$, i.\,e.\ 
  for all dg modules $M$, the canonical map $\Hom_{\dgHot}({P}, M) \ra
  \Hom_{\dgDer}({P},M)$ is an isomorphism.
\item For each acyclic dg module ${M}$, we have $\Hom_{\dgHot}({P},
  {M})=0$.  
\end{enumerate}
In \cite[3.1]{Keller-deriving-dg-cat} such a module is said to have \stress{property (P)},
in \cite[10.12.2]{BL} the term \stress{$\mathcal{K}$-projective} is used.
Since evaluation at $1 \in A^0$ is an isomorphism
\begin{equation}
  \label{eq:Ahopro}
  \Hom_{\dgHot}(\mathcal{A}, {M}) \sira H^0(M),
\end{equation}
$\mathcal{A}$ and each direct summand
of $\mathcal{A}$ is homotopically projective.

Let $\dgHotproj(\mathcal{A})$ be the full subcategory of $\dgHot(\mathcal{A})$
consisting of homotopically projective dg modules (this is the category 
$\mathcal{H}_p$ in \cite[3.1]{Keller-deriving-dg-cat}).
It is a triangulated subcategory of $\dgHot(\mathcal{A})$
and closed under taking direct summands.
The quotient functor
$\dgHot(\mathcal{A}) \ra \dgDer(\mathcal{A})$
induces a triangulated
equivalence (\cite[3.1, 4.1]{Keller-deriving-dg-cat}) 
\begin{equation}
  \label{eq:dgHotp-equiv-dgDer}
  \dgHotproj(\mathcal{A}) \sira \dgDer(\mathcal{A}).
\end{equation}
The category $\dgPer(\mathcal{A})$ ($\dgPerDer(\mathcal{A})$ respectively) 
is defined to
be the smallest full triangulated subcategory of $\dgHot(\mathcal{A})$
(of $\dgDer(\mathcal{A})$ respectively)
that contains $\mathcal{A}$ and is closed under taking direct summands. 
The category $\dgPer(\mathcal{A})$ of perfect dg modules is a subcategory of $\dgHotproj(\mathcal{A})$.
The category $\dgPerDer(\mathcal{A})$ is called the perfect derived
category of $\mathcal{A}$.
Equivalence \eqref{eq:dgHotp-equiv-dgDer} restricts to a triangulated equivalence
\begin{equation*}
  \dgPer(\mathcal{A}) \sira \dgPerDer(\mathcal{A}).
\end{equation*}
In the following we prefer to work in the homotopy category
$\dgHot(\mathcal{A})$ and leave it to the reader to transfer our
results from $\dgPer(\mathcal{A})$ to 
$\dgPerDer(\mathcal{A})$.
Note that the objects of $\dgPer(\mathcal{A})$ are precisely the
homotopically projective objects of $\dgPerDer(\mathcal{A})$.

\begin{remark}
  \label{rem:conc-degree-zero}
  Let $\mathcal{B}$ be a dg algebra concentrated in degree zero,
  i.\,e.\ a dg algebra whose underlying graded algebra is concentrated
  in degree zero. 
  Then $\mathcal{B}$ is necessarily of the form
  $\mathcal{B}=(B=B^0,d=0)$ for some algebra $B$.

  In this case $\dgMod(\mathcal{B})$ is the category of complexes in
  $\Mod(B)$, $\dgHot(\mathcal{B})$ is the corresponding homotopy
  category and $\dgDer(\mathcal{B})$ is the derived category of the
  abelian category $\Mod(B)$. 
  The objects of the category $\dgPer(\mathcal{B})$ are the objects of
  $\dgHotproj(\mathcal{B})$ that are isomorphic to bounded complexes of
  finitely generated projective $B$-modules (see 
  for example \cite[Prop.~3.4]{boekstedt-neeman-homotopy}).
  \remarkend
\end{remark}

\section{Filtered DG Modules}
\label{sec:filtered-dg-modules}

In this chapter we introduce a certain category $\dgFilt$ of filtered
dg modules (for a suitable dg algebra). Later on we will see that this category is equivalent to
$\dgPer$.

Let $\mathcal{A}=(A=\bigoplus_{i \in \DZ} A^i, d)$ be a dg algebra. 
In the rest of this article we always assume that $\mathcal{A}$ satisfies the
following conditions:
\begin{enumerate}[label={(P\arabic*)}]
\item 
\label{enum:pg}
$A$ is positively graded, i.\,e.\ $A^i=0$ for $i < 0$;
\item 
\label{enum:ss}
$A^0$ is a semisimple ring;
\item 
\label{enum:sdga}
the differential of $\mathcal{A}$ vanishes on $A^0$, i.\,e.\ $d(A^0)=0$. 
\end{enumerate}

Then $A^0$ has only a finite number
of non-isomorphic simple (right) modules $(L_x)_{x \in W}$,
and $A^0$ is a dg subalgebra $\mathcal{A}^0$ of $\mathcal{A}$.

As in the graded setting,
the inclusion $\mathcal{A}^0 \hra \mathcal{A}$ and the projection
$\mathcal{A} \ra \mathcal{A}/\mathcal{A}^+=\mathcal{A}^0$ give rise to
extension of scalars functors
\begin{align*}
\dgMod(\mathcal{A}) \ra\dgMod(\mathcal{A}^0), \quad &M \mapsto
\ol{M}:=M/M\mathcal{A}^+=M\otimes_{\mathcal{A}} \mathcal{A}^0,\\
\dgMod(\mathcal{A}^0) \ra \dgMod(\mathcal{A}), \quad &M \mapsto \hat M := M \otimes_{\mathcal{A}^0} \mathcal{A}.
\end{align*}

We often view $A^0$-modules as dg $\mathcal{A}^0$-modules concentrated
in degree zero. In this manner, we obtain $\mathcal{A}$-modules
$\hat L_x=L_x \otimes_{\mathcal{A}^0} \mathcal{A}$. Since each $L_x$
is a direct summand of the $A^0$-module $A^0$, each $\hat L_x$ is a
direct summand of the dg module $A^0 \otimes_{\mathcal{A}^0} \mathcal{A}=
\mathcal{A}$.

\begin{examples}
  \label{exs:main}
  Let $R$ be a semisimple ring (and a $k$-algebra). Then 
  the dg algebra $(R=R^0,d=0)$ satisfies the conditions
  \ref{enum:pg}-\ref{enum:sdga}. Any dg algebra that is concentrated
  in degree zero and satisfies condition \ref{enum:ss} is of this
  form. 

  Let $A=\bigoplus_{i \in \DN} A^0$ be a positively
  graded algebra with $A^0$ a semisimple ring. Then the dg algebra
  $(A, d=0)$ satisfies the conditions \ref{enum:pg}-\ref{enum:sdga}.
  Any dg algebra with vanishing differential that satisfies conditions
  \ref{enum:pg}-\ref{enum:ss} is of this form.
  For example $A$ could be a polynomial algebra 
  $k[X_1,\dots, X_n]$ over a field $k$ with homogeneous generators
  $X_i$ of strictly positive degrees. In this case there is only one
  simple $k$-module $L=k$, and $\hat{L}=k[X_1,\dots, X_n]$.
  A more general example for such an $A$ would be a quiver algebra
  (over a field $k$) of a quiver with finitely many vertices and
  arrows of strictly positive degrees, or a quotient of such an
  algebra by a homogeneous ideal (which is assumed to be zero in degree zero). 
  Each vertex $x$ of the quiver gives rise to an idempotent $e_x$ in
  $A$ and to a dg module $\hat L_x=e_xA$.
    
  Examples of dg algebras satisfying conditions
  \ref{enum:pg}-\ref{enum:sdga} with non-vanishing differential arise
  for example in rational homotopy theory (see
  \cite{DGMS-real-homotopy} or \cite[\S 19]{Bott-Tu-diff-forms}). 
  \examplesend
\end{examples}

We consider the following full subcategories of $\dgHot(\mathcal{A})$:
\begin{itemize}
\item $\dgFilt(\mathcal{A})$: Its objects are $\mathcal{A}$-modules
  $M$ admitting a finite filtration 
  $0 = F_0(M) \subset F_1(M) \subset \dots \subset
  F_n(M) =M$ by dg submodules with subquotients
  \begin{equation*}
    F_{i}(M)/F_{i-1}(M) \cong \{l_i\}\hat L_{x_i}
    \quad\text{in $\dgMod(\mathcal{A})$}
  \end{equation*}
  for suitable $l_1 \geq l_2 \geq \dots \geq l_n$ and $x_i \in W$.
  We call such a filtration an \define{$\hat L$-filtration}.
\item $\dgPrae$: This is the smallest strict (= closed under
  isomorphisms) full triangulated subcategory of $\dgHot(\mathcal{A})$
  that contains all objects $(\hat L_x)_{x \in W}$.
\end{itemize}
These two subcategories correspond under the equivalence
\eqref{eq:dgHotp-equiv-dgDer} to the categories
$\dgFiltDer(\mathcal{A})$ and $\dgPraeDer(\mathcal{A})$ of the introduction.

\begin{remark}
  \label{rem:objects-dgfilt}
  The condition $l_1 \geq l_2 \geq \dots \geq l_n$ is essential for
  the definition of objects of $\dgFilt(\mathcal{A})$. 
  It means that $F_1(M) \cong \{l_1\}\hat L_{x_1}$ is generated in degree $-l_1$ (as a
  graded $A$-module), then $F_2(M)$ is an extension of $\{l_2\}\hat L_{x_2}$ (which is generated in
  degree $-l_2 \geq -l_1$) by $F_1(M)$ and so on.
  Since $\{l_2\}\hat L_{x_2}$ is projective as a graded $A$-module, we
  obtain $F_2(M)\cong \{l_1\}\hat L_{x_1} \oplus \{l_2\}\hat L_{x_2}$
  in $\gMod(A)$. The same reasoning yields by induction 
  \begin{equation}
    \label{eq:dgFilt-as-graded}
    F_i(M)\cong \{l_1\}\hat L_{x_1} \oplus \dots \oplus \{l_i\}\hat
    L_{x_i} \quad \text{in $\gMod(A)$};
  \end{equation}
  note that the differential of $F_i(M)$ does not necessarily respect
  such a direct sum decomposition.

  Let us draw some modules in $\dgFilt(\mathcal{A})$.
  We picture a module $\hat L_x\{l\}$ (with $x \in W$, $l \in \DZ$)
  as follows:
  \begin{equation}
    \label{eq:picture-hatL}
    \xymatrix@R=0pt@C=15pt{
      {\text{degree:}} & {-l-2} & {-l-1} & {-l} & {-l+1} & {-l+2} &
      {-l+3} & {\dots} \\ 
      {\hat L_{x}\{l\}:} & {} & {} & {\bullet} & {\circ} \ar[r] & 
      {\circ} \ar[r] & {\circ} \ar[r] & {\dots} 
      }
  \end{equation}
  A white ($\circ$) or black ($\bullet$) bead in a column labeled $i$
  represents $(L_x\{l\})^i$. If 
  there is no bead in column $i$, we have
  $(L_x\{l\})^i=0$. Since $L_x\{l\}$ is generated by its component in degree $-l$ (the
  black bead),
  there are only beads in degrees $\geq -l$.
  (The action of $A$ is not explicitly drawn in this picture: Elements
  of $A^j$ map elements of a bead to the bead that is $j$ steps to the right.)
  There is an arrow between two beads if the differential
  between the corresponding components is possibly $\not= 0$. Since
  the differential has degree $1$, all arrows go from a column to its
  right neighbour.
  Note that there is no arrow starting at the black bead since
  $\hat L_x$ is a direct summand of $\mathcal{A}$ and $d(A^0)=0$.
  
  Now we can draw pictures of more general objects. For example, let
  $M$ be an object of
  $\dgFilt(\mathcal{A})$ that admits an $\hat
  L$-filtration with four steps $0 = F_0 \subset F_1 \subset F_2 
  \subset F_3 \subset F_4=M$ and subquotients 
   \begin{equation*}
     F_1 \cong \{2\}\hat L_{x_1}, \qquad F_2/F_1 \cong \{1\}\hat L_{x_2}, \quad
     F_3/F_2 \cong \{1\}\hat L_{x_3}, \quad F_4/F_3 \cong \{-2\}\hat L_{x_4}
   \end{equation*}
  for some $x_i \in W$.
  As in \eqref{eq:dgFilt-as-graded} we 
  identify $M=F_4$ as a graded $A$-module with 
  $\{2\}\hat L_{x_1} \oplus \{1\}\hat L_{x_2} \oplus \{1\}\hat L_{x_3}
  \oplus \{-2\}\hat L_{x_4}$
  such that $F_i$ gets identified with the first $i$ summands.
  Here is our picture of $M$:
  \begin{equation}
    \label{eq:picture-dgfilt}
    \xymatrix@R=6pt@C=15pt{
      {\text{degree:}} & {-3} & {-2} & {-1} & {0} & {1} & {2} &
        {3} & {4} & {5} & {\dots} \\
      {\{-2\}\hat L_{x_4}:} & {} & {} & {} & {} & {} & {\bullet} \ar[rd] \ar[rdd] \ar[rddd] & {\circ} \ar[r] \ar[rd] \ar[rdd] \ar[rddd] &
      {\circ} \ar[r] \ar[rd] \ar[rdd] \ar[rddd] & {\circ} \ar[r] \ar[rd] \ar[rdd] \ar[rddd] &
      {\dots} \\ 
      {\{1\}\hat L_{x_3}:} & {} & {} & {\bullet} \ar[rd] \ar[rdd] &
      {\circ} \ar[r] \ar[rd] \ar[rdd] & {\circ}
      \ar[r] \ar[rd] \ar[rdd] & {\circ} \ar[r] \ar[rd] \ar[rdd] & {\circ} \ar[r]
      \ar[rd] \ar[rdd] & {\circ} \ar[r] \ar[rd] \ar[rdd] & {\circ} \ar[r] \ar[rd] \ar[rdd] &
      {\dots} \\ 
      {\{1\}\hat L_{x_2}:} & {} & {} & {\bullet} \ar[rd] &
      {\circ} \ar[r] \ar[rd] & {\circ}
      \ar[r] \ar[rd] & {\circ} \ar[r] \ar[rd] & {\circ} \ar[r]
      \ar[rd] & {\circ} \ar[r] \ar[rd] & {\circ} \ar[r] \ar[rd] &
      {\dots} \\ 
      {\{2\}\hat L_{x_1}:} & {} & {\bullet} & {\circ} \ar[r] & 
      {\circ} \ar[r] & {\circ} \ar[r] &
      {\circ} \ar[r] & {\circ} \ar[r] & {\circ} \ar[r] &
      {\circ} \ar[r] & {\dots} 
      }
  \end{equation}
  The direct sum of the first $i$ rows from below (without the arrows)
  represents the module $F_i$ as a
  graded $A$-module.
  The arrows represent the differential of $M$: Elements of a bead $b$
  are mapped to the direct sum of those beads that are endpoints of
  arrows starting at $b$.

  Since $M$ is generated by the black beads as an $A$-module, the dg
  submodule $M\mathcal{A}^+$ is represented by the white beads, 
  and the quotient module $\ol{M}=M/M\mathcal{A}^+$ is represented by
  the black beads (without arrows): The differential on $\ol{M}$
  vanishes.
  \remarkend
\end{remark}

\begin{remark}
  \label{rem:semisimple}
  Let $\mathcal{R}=(R=R^0, d=0)$ be a dg algebra concentrated in
  degree zero.
  Recall from Remark \ref{rem:conc-degree-zero} that each
  object of $\dgPer(\mathcal{R})$ is isomorphic to a bounded complex of
  finitely generated projective $R$-modules. 
  Let $\mathcal{M}=(M=\bigoplus_{j \in \DZ} M^j, d)$ be such a complex. 
  Assume that $M^j=0$ for $|j|\geq N$. Define subcomplexes $G_i(M)$ of $M$
  by
  \begin{equation*}
    G_i(M)^j =
    \begin{cases}
      M^j & \text{if $j \geq N-i$,}\\
      0 & \text{otherwise.}
    \end{cases}
  \end{equation*}
  This defines a finite filtration $0=G_0(M) \subset G_1(M) \subset
  \dots \subset G_{2N}(M)=M$ of $M$ with subquotients
  $G_i(M)/G_{i-1}(M) \cong \{p_i\}M^{N-i}$ with $p_i=i-N$ (here we
  consider $M^{N-i}$ as an $R$-module sitting in degree zero). Note that
  $p_1 \leq p_2 \leq \dots \leq p_{2N}$.

  Now assume that $R$ is a semisimple ring with simple modules $(L_x)_{x \in W}$;
  then we obtain from the above filtration $(G_i(M))_i$ a finite increasing filtration
  $(G_i'(M))_{i=0}^n$ with subquotients 
  $G_i'(M)/G_{i-1}'(M) \cong \{q_i\}L_{x_i}$ with 
  $q_1 \leq q_2 \leq \dots \leq p_{n}$ and $x_i \in W$.
  
  So we have opposite inequalities compared to the definition of the
  category $\dgFilt(\mathcal{R})$. The objects of
  $\dgFilt(\mathcal{R})$ are precisely the bounded complexes of finitely
  generated (projective) $R$-modules with differential zero (cf.\
  Picture \eqref{eq:picture-dgfilt})
  
  Nevertheless the inclusion $\dgFilt(\mathcal{R}) \subset
  \dgPer(\mathcal{R})$ is an equivalence of categories: This is well
  known since $R$ is a semisimple ring by assumption. We will prove this
  for arbitrary $\mathcal{A}$ satisfying \ref{enum:pg}-\ref{enum:sdga}
  in Theorems  \ref{t:filt-iso-prae} and \ref{t:t-cat-prae}.
  \remarkend
\end{remark}

\begin{lemma}
  \label{l:inclusions}
  We have inclusions
  \begin{equation*}
    \satzdgFilt \subset \satzdgPrae \subset \satzdgPer \subset \dgHotproj \subset
    \dgHot.
  \end{equation*}
  If $(M,d)$ is in $\satzdgFilt$, the underlying graded module $M$ is in $\gProjf(A)$.
\end{lemma}
\begin{remark}
  We prove later on that the inclusion $\dgFilt \subset \dgPrae$ is an
  equivalence (cf.\ Theorem \ref{t:filt-iso-prae}).
  Moreover, $\dgPrae$ is in fact closed under forming direct summands, which 
  implies $\dgPrae = \dgPer$ (cf.\ Theorem \ref{t:t-cat-prae}).
\remarkend
\end{remark}
\begin{proof}
As a right module over itself, $A^0$ is isomorphic to a finite direct
sum of simple modules $L_x$. Hence $\mathcal{A}$ is isomorphic to a finite direct
sum of modules $\hat L_x$. In particular, each $\hat L_x$ is in
$\dgPer$ and homotopically
projective (see \eqref{eq:Ahopro}).
This shows $\dgPrae \subset \dgPer$.

As a graded $A$-module, each $\hat L_x$ is projective. So any $\hat
L$-filtration of a dg module $(M,d)$ yields several $A$-split short 
exact sequences in $\dgMod$; in particular the graded module $M$ is
in $\gProjf(A)$ (as already seen in Remark \ref{rem:objects-dgfilt}).
Moreover these $A$-split exact sequences become distinguished
triangles in $\dgHot$ and show that every object of $\dgFilt$ lies in
$\dgPrae$. 

The remaining inclusions are obvious from the definitions.
\end{proof}

Let $N$ be a module over a ring. 
If $N$ has a composition series (of finite length) we denote its length by $\lambda(N)$;
otherwise we define $\lambda(N)=\infty$.
If $(M,d)$ is in $\dgFilt$, $\ol M$ is a finitely generated module
over the semisimple ring $A^0$, and its length $\lambda(\ol M)$ obviously
coincides with the length of any $\hat 
L$-filtration of $(M,d)$. 

If $M$ is a graded $A$-module, we 
define 
\begin{equation*}
  M_{\leq i} := \sum\nolimits_{j \leq i} M^j A
\end{equation*}
to
be the graded submodule of $M$ that is generated by the degree $\leq i$
parts. This defines an increasing filtration $\ldots \subset M_{\leq
  i}\subset M_{\leq i+1}\subset \dots$ of $M$.
If $M$ is the underlying graded module of a dg module in $\dgFilt$, 
the different entries in this filtration define a filtration that is coarser than any $\hat{L}$-filtration.
Define $M_{<i}:= M_{\leq i-1}$.

\begin{theorem}
  \label{t:filt-iso-prae}
  Assume that $\mathcal{A}$ is a dg algebra satisfying \ref{enum:pg}-\ref{enum:sdga}.
  Then every object of $\satzdgPrae$ is isomorphic to an object
  of $\satzdgFilt$. 
  In other words: The inclusion $\satzdgFilt \subset \satzdgPrae$ is an
  equivalence of categories.
\end{theorem}

\begin{proof}
  All generators $\hat{L}_x$ of $\dgPrae$ are in $\dgFilt$,
  and $\dgFilt$ is closed under the shift $\{1\}$. 
  So it is sufficient to prove:

  \textbf{Claim:} Let $f: X \ra Y$ be a morphism
  in $\dgMod$ of objects $X$, $Y$ of
  $\dgFilt$. Then the cone $C(f)$ is isomorphic in $\dgHot$ to an object of
  $\dgFilt$.

  We will prove this by induction on the length $\length(\ol{X})$ of
  $\ol X$. We may assume
  that $\length(\ol{X}) > 0$.

  Recall that the cone $C(f)$ is given by the graded (right)
  $A$-module $Y\oplus\{1\} X$ with differential
  $d_{C(f)}=\tzmat{d_Y}f0{-d_X}$. 

  \textbf{The case $\length(\ol{X}) = 1$:}
  Using the shift $\{1\}$ we may assume that $X=\hat{L}_x$ for some $x \in W$.
  Choose an $\hat L$-filtration $(F_i(Y))_{i=0}^n$ of $Y$ and abbreviate
  $F_i=F_i(Y)$. Assume that $F_{i}/F_{i-1} \cong \{l_i\}\hat{L}_{x_i}$,
  for $l_1 \geq l_2 \geq \dots \geq l_n$ and $x_i \in W$.

  If the image of $f$ is contained in $Y_{< 0}$, 
  let $s \in \{0, \dots, n\}$ 
  with $F_s= Y_{<0}$.
  Then 
  \begin{equation*}
    G_i = 
    \begin{cases}
      F_i\oplus 0 & \text{if $i \leq s$,}\\
      F_{i-1}\oplus \{1\}\hat{L}_x & \text{if $i>s$.}
    \end{cases}
  \end{equation*}
  defines an 
  $\hat L$-filtration $(G_i)_{i=0}^{n+1}$ 
  of the cone $C(f)=Y\oplus\{1\}\hat{L}_x$. Thus $C(f)$ is in $\dgFilt$.

  Now assume $\bild f \not \subset Y_{<0}$. 
  Let $t \in \{0, n\}$ be minimal with $\bild f \subset F_t$.
  Then $t \geq 1$ and $l_t= 0$, since $X$ is generated by its degree zero part.
  The composition
  \begin{equation*}
    X=\hat{L}_x \xra{f} F_t \sra F_t/F_{t-1} \cong \hat{L}_{x_t}
  \end{equation*}
  is non-zero and an isomorphism by Lemma \ref{l:zero-or-iso}.
  Hence $f$ induces an isomorphism $f:\hat{L}_x \sira \bild f$, and
  its mapping cone $V=(\bild f)\oplus\{1\}\hat{L}_x$ is an acyclic dg
  submodule of $C(f)$. 

  The inverse of the isomorphism $\bild f \sira F_t/F_{t-1}$ splits the short exact
  sequence $(F_{t-1}, F_{t}, F_t/F_{t-1})$.
  This implies that, for $1 \leq i \leq n$, 
  \begin{equation}
    \label{eq:filtschnitt}
    F_i \cap (F_{i-1}+\bild f) =
    \begin{cases}
      F_{i-1} & \text{if $i \not= t$,}\\
      F_i & \text{if $i=t$.}
    \end{cases}
  \end{equation}

  The filtration $(F_i+V)_i$ of $C(f)$ induces a filtration of 
  $C(f)/V$ with successive subquotients
  \begin{equation*}
    \frac{F_i+V}{F_{i-1}+V} \sila \frac{F_i}{F_i\cap (F_{i-1}+V)}
    \sila \frac{F_i}{F_i\cap (F_{i-1}+\bild f)}
  \end{equation*}
  Using \eqref{eq:filtschnitt}, we see that this filtration is an $\hat L$-filtration of
  $C(f)/V$, hence $C(f)/V$ is in $\dgFilt$.

  The short exact sequence $(V, C(f), C(f)/V)$ of
  $\mathcal{A}$-modules and the acyclicity of $V$ show that 
  $C(f) \ra C(f)/V$ is a quasi-isomorphism. 
  Since both $C(f)$ and $C(f)/V$ are homotopically projective, it is
  an isomorphism in $\dgHot$. Hence $C(f)$ is isomorphic to an object
  of $\dgFilt$.

  \textbf{Reduction to the case $\length(\ol{X}) = 1$:}
  This follows from \cite[1.3.10]{BBD} but let me include the proof
  for convenience.

  Let $U:=F_1(X)$ be the first step of an $\hat L$-filtration of $X$,
  $u: U \ra X$ the inclusion and $p:X\ra X/U$ the projection. 
  Then the short exact sequence $0\ra U \xra{u} X \xra{p} X/U \ra 0$ is
  $A$-split and defines a distinguished triangle $(U,X,X/U)$ in $\dgHot$.
  We apply the octahedral axiom to the maps $u$ and $f$ and get the
  dotted arrows in the following commutative diagram.
  \begin{equation*}
    \xymatrix@dr{
      & {\{1\}U} \ar[r]^{\{1\}u} 
      & {\{1\}X} \ar[r]^{\{1\}p} 
      & {\{1\}X/U} \\
      {X/U} \ar@(ur,ul)[ru] \ar@{..>}[r] 
      & {C(f\comp u)} \ar@{..>}[r] \ar[u] 
      & {C(f)} \ar@(dr,dl)[ru] \ar[u]\\
      X \ar[u]^p \ar[r]^f 
      & Y \ar@(dr,dl)[ru] \ar[u]\\
      U \ar[u]^u \ar@(dr,dl)[ru]^{f\comp u}
    }
  \end{equation*}
  The four paths with the bended arrows are distinguished
  triangles.

  By the length 1 case we may replace $C(f \comp u)$ by an isomorphic
  object $\tilde C(f \comp u)$ of $\dgFilt$. 
  So $C(f)$ is isomorphic to the cone of a morphism 
  $X/U \ra \tilde C(f \comp u)$ and hence, by induction, isomorphic to
  an object of $\dgFilt$.
\end{proof}

\section{t-Structure}
\label{sec:t-structure}

Using the equivalence $\dgFilt \subset \dgPrae$, we define a
bounded t-structure (see \cite{BBD}) on $\dgPrae$. As a corollary we obtain that
$\dgPrae$ coincides with $\dgPer$.   
Let $\mathcal{A}$ as before satisfy \ref{enum:pg}-\ref{enum:sdga}.

If $M$ is a graded $A^0$-module, we define its support $\supp M$ by
\begin{equation*}
  \supp M := \{i \in \DZ \mid M^i \not= 0\}
\end{equation*}
If $I \subset \DZ$ is a subset we define the full subcategory
\begin{equation}
  \label{eq:dgperI}
  \dgPer^{I} := \{{M} \in \dgPer \mid 
  \supp {H(\ol{M})} \subset I\}.
\end{equation}
By replacing $\dgPer$ by $\dgPrae$ or $\dgFilt$, we define $\dgPrae^I$
and $\dgFilt^I$.
We write $\dgPer^{\leq n}$, $\dgPer^{\geq n}$, $\dgPer^n$ instead of
$\dgPer^{(-\infty, n]}$, $\dgPer^{[n,\infty)}$, $\dgPer^{[n,n]}$ respectively, and
similarly for $\dgPrae$ and $\dgFilt$.

\begin{remark}
In \eqref{eq:dgperI} we could also write 
${M} \Lotimes_{\mathcal{A}} \mathcal{A}^0$
instead of 
$\ol M={M} \otimes_{\mathcal{A}} \mathcal{A}^0$, 
since each object of $\dgPer$ is homotopically projective, hence
homotopically flat (here $(? \Lotimes_{\mathcal{A}}
\mathcal{A}^0):\dgDer(\mathcal{A}) \ra \dgDer(\mathcal{A}^0)$ is the
derived functor of the extension of scalars functor).
In fact this is the correct definition if one works in $\dgDer$ instead
of $\dgHotproj$ (as we do in the introduction).
\remarkend
\end{remark}

\begin{remark}
  \label{rem:tstruktur-dgfilt}
  It is instructive to consider an object ${M}$ of $\dgFilt$ (and to
  have a picture in mind as picture \eqref{eq:picture-dgfilt} in Remark
  \ref{rem:objects-dgfilt}).
  Note that the differential of $\ol{M}$ vanishes: 
  This is explained at the end of Remark \ref{rem:objects-dgfilt}
  in an example, but the argument generalizes immediately to an
  arbitrary object of $\dgFilt$.
  This implies that $\ol{M}$ and $H(\ol{M})$ coincide and in
  particular have the same support.
  So ${M}$ lies in
  $\dgPrae^{\leq n}$ (in $\dgPrae^{\geq n}$) if and only if it is
  generated in degrees $\leq n$ (in degrees $\geq n$)
  as a graded $A$-module.
\remarkend
\end{remark}

\begin{theorem}[t-structure]
  \label{t:t-cat-prae}
  Let $\mathcal{A}$ be a dg algebra satisfying \ref{enum:pg}-\ref{enum:sdga}.
  Then $\satzdgPrae=\satzdgPer(\mathcal{A})$ and
  $(\satzdgPer^{\leq 0}, \satzdgPer^{\geq 0})$
  defines a bounded 
  (in particular non-degenerate) 
  t-structure on $\satzdgPer(\mathcal{A})$.
\end{theorem}

\begin{remark}
  In the special case that $\mathcal{A}$ is a polynomial ring
  $(\DR[X_1, \dots, X_n], d=0)$ with generators in strictly positive
  even degrees and differential zero, this t-structure
  coincides with the t-structure defined in \cite[11.4]{BL}.
\remarkend
\end{remark}

\begin{proof}
  We first prove that $(\dgPrae^{\leq 0}, \dgPrae^{\geq 0})$
  defines a bounded t-struc\-ture on $\dgPrae$ and have to check
  the three defining properties (\cite[1.3.1]{BBD}). 
\begin{enumerate}
\item $\Hom_{\dgHot}(X, Y)=0$ for $X \in \dgPrae^{\leq 0}$ and $Y \in
  \dgPrae^{\geq 1}$:
  By Theorem \ref{t:filt-iso-prae} we may assume that $X$, $Y$ are
  in $\dgFilt$. But then even any morphism of the underlying
  graded $A$-modules is zero.
\item $\dgPrae^{\leq 0} \subset
  \dgPrae^{\leq 1}$ and $\dgPrae^{\geq 1} \subset
  \dgPrae^{\geq 0}$: Obvious.
\item\label{en:triangle-t-prae} If $X$ is in $\dgPrae$ there is a
  distinguished triangle 
  $(M, X, N)$ with $M$ in
  $\dgPrae^{\leq 0}$ and $N$ in $\dgPrae^{\geq 1}$:
  We may assume that $X$ is in $\dgFilt$. The graded $A$-submodule
  $X_{\leq 0}$ of $X$ is an
  $\mathcal{A}$-submodule, since it appears in any
  $\hat L$-filtration of $X$. The short exact sequence 
  $X_{\leq 0} \ra X \ra X/X_{\leq 0}$ 
  is $A$-split and hence defines a distinguished triangle $(X_{\leq 0}, X,
  X/X_{\leq 0})$ in $\dgHot$. All terms of this triangle are in
  $\dgFilt$, $X_{\leq 0}$ is in $\dgPrae^{\leq 0}$ and 
  $X/X_{\leq 0}$ in $\dgPrae^{\geq 1}$.
\end{enumerate}

It is obvious that any object of $\dgFilt$ is contained in 
$\dgFilt^{[a,b]}$, for some integers $a \leq b$.
Hence our t-structure is bounded.

We claim that $\dgPrae \subset \dgHot(\mathcal{A})$
is closed under taking direct summands. This will imply that
$\dgPrae=\dgPer$. But our claim is an application of the main result
of  \cite{le-chen-karoubi-trcat-bdd-t-str}: A triangulated category
with a bounded t-structure is Karoubian (= idempotent-split = idempotent complete)
(cf. \cite{balmer-schlichting}).
\end{proof}

The t-structure $(\dgPer^{\leq 0}, \dgPer^{\geq 0})$ on $\dgPer$
yields truncation functors 
$\tau_{\leq n}$ and $\tau_{\geq n}$ (\cite[1.3.3]{BBD}).
For objects $M$ in the equivalent subcategory $\dgFilt$, we can assume
that these truncation functors are given by $M \mapsto M_{\leq n}$ 
and $M \mapsto M/M_{< n}$ (and similarly for morphisms). 
Note that these objects are again in $\dgFilt$. 
Hence the truncation functors are very explicit on $\dgFilt$.

\section{Heart}
\label{sec:heart}
We show that the heart of our t-structure on $\dgPer(\mathcal{A})$ is
naturally equivalent to a full abelian subcategory of
$\dgMod(\mathcal{A})$. 
We keep the assumption that $\mathcal{A}$ satisfies \ref{enum:pg}-\ref{enum:sdga}.

Let $\heart=\dgPer^{0}$ be the heart
of our t-structure 
on $\dgPer$.
Recall that $\dgFilt^{0}$ is the full subcategory of $\dgFilt$ consisting of
objects $M$ with $\supp \ol M \subset \{0\}$ 
(cf.\ Remark \ref{rem:tstruktur-dgfilt}). 
The underlying graded $A$-module of such an object is isomorphic to a finite
direct sum of some $\hat{L}_x$ (without shifts).

\begin{remark}
  The objects of $\dgFilt^0$ share some similarity with the so-called
  ``linear complexes'' (see e.\,g.\ \cite[Sect.~3]{MOS-quadratic} and
  references therein). Let me explain the relation informally: A
  linear complex is a certain complex of graded modules. If
  one takes some kind of total complex, one obtains an object that
  looks like an object of $\dgFilt^0$. The differential of a general object of
  $\dgFilt^0$ however can be more complicated.
  For example, the complex $P$ and the dg module $\mathcal{K}(A)$ in Chapter
  \ref{sec:koszul-duality} are examples of a linear complex and the
  associated object of $\dgFilt^0$.
  It might be interesting to consider common generalizations
  of both notions.
  \remarkend
\end{remark}

From the above discussion it is clear
that $\dgFilt^0$ is contained in $\heart$ and that each object of
$\heart$ is isomorphic to an object of $\dgFilt^0$. This shows

\begin{proposition}
  \label{p:dgfiltzero-heart}
  The inclusion $\satzdgFilt^0 \subset \heart$ is an equivalence of
  categories. In particular, $\satzdgFilt^0$ is an
  abelian category.
\end{proposition}

Let $\dgFilMod$ be the full subcategory of
$\dgMod$ with the same objects as $\dgFilt^0$.
We will show in Propositions \ref{p:dgfilmod-abelian-sub} and
\ref{p:dgfilmod-equi} that $\dgFilMod$ is a full abelian subcategory
of $\dgMod$ and that the obvious functor $\dgFilMod \ra \dgFilt^0$ is an
equivalence. Hence the abelian structure on $\dgFilt^0$ is the 
most obvious one.

\begin{proposition}
  \label{p:dgfilmod-abelian-sub}
  The category $\dgFilMod$ is abelian and the inclusion
  of $\dgFilMod$ in $\dgMod$ is exact.
  In short, $\dgFilMod$ is a full abelian subcategory of $\dgMod$. 
\end{proposition}

We need the following remark in the proof.
\begin{remark}\label{rem:essim}
  We view the extension of scalars functor 
  $\pro_{A^0}^A$ from \eqref{eq:prodhat} as a functor 
  $\pro_{A^0}^A: \Mod(A^0) \ra \gMod(A)$.
  Since $\pro_{A^0}^A$ is exact and fully faithful,
  its essential image $\mathcal{B}$ is closed under taking kernels,
  cokernels, images and pull-backs. In
  particular, if $U$ and $V$ are graded submodules of $M$ with $U$, $V$ and
  $M$ in $\mathcal{B}$, then $U \cap V$ and $U+V$ are in $\mathcal{B}$.
\remarkend
\end{remark}

\begin{proof}[Proof of Proposition \ref{p:dgfilmod-abelian-sub}]
  Let $f: \mathcal{M} \ra \mathcal{N}$ be a morphism in $\dgFilMod$. 
  Let $\mathcal{K}=(K,d)$ and $\mathcal{Q}$ be the kernel and cokernel of $f$
  in the abelian category $\dgMod$.
  It is sufficient to prove that $\mathcal{K}$ and $\mathcal{Q}$
  belong to $\dgFilMod$.

  Since the underlying graded $A$-module of an object of $\dgFilMod$
  is in the essential image $\mathcal{B}$ of
  $\pro_{A^0}^A$, the kernel $K$ lies in $\mathcal{B}$ by Remark
  \ref{rem:essim}. 
    
  Let $\mathcal{I}\subset \mathcal{N}$ be the image of $f$ in
  $\dgMod$. Since $\mathcal{M}/\mathcal{K} \sira \mathcal{I}$ and
  $\mathcal{Q} \sira \mathcal{N}/\mathcal{I}$, it is sufficient to
  prove:

  \textbf{Claim:} If $\mathcal{U}=(U,d_U) \subset \mathcal{M}$ is an
  $\mathcal{A}$-submodule of $\mathcal{M} \in \dgFilMod$ with
  $U$ in $\mathcal{B}$, then 
  $\mathcal{U}$ and the quotient $\mathcal{M}/\mathcal{U}$ in $\dgMod$
  are objects of $\dgFilMod$.

  \textbf{Proof of the claim:} Let $(F_i(\mathcal{M}))$ be an
  $\hat L$-filtration of $\mathcal{M}$ with subquotients
  $F_i(\mathcal{M})/F_{i-1}(\mathcal{M})\cong \hat{L}_{x_i}$, and  
  $(F_i(\mathcal{U}))$ and $(F_i(\mathcal{M}/\mathcal{U}))$ the induced filtrations of
  $\mathcal{U}$ and $\mathcal{M}/\mathcal{U}$. 
  The underlying graded modules of all steps of all these filtrations are
  in $\mathcal{B}$: For $F_i(\mathcal{M})$ this is obvious, and for
  $F_i(\mathcal{U})$ and $F_i(\mathcal{M}/\mathcal{U})$ this is a
  consequence of Remark \ref{rem:essim}.
  Consider the sequence of finitely filtered objects
  \begin{equation*}
    0 
    \ra (\mathcal{U}, F_i(\mathcal{U})) 
    \xra{\upsilon} (\mathcal{M}, F_i(\mathcal{M})) 
    \xra{\pi} (\mathcal{M}/\mathcal{U}, F_i(\mathcal{M}/\mathcal{U}))
    \ra 0.
  \end{equation*}
  Let $\Gr_i(\mathcal{M})=F_i(\mathcal{M})/F_{i-i}(\mathcal{M})$ be
  the $i$-th component of the associated graded object of the filtered
  object $(\mathcal{M}, (F_i(\mathcal{M})))$, and similarly for other
  filtered objects and morphisms.
  Since $\mathcal{U}$ and
  $\mathcal{M}/\mathcal{U}$ are equipped with the induced filtrations,
  we obtain short exact sequences 
  \begin{equation*}
    0 
    \ra \Gr_i(\mathcal{U})
    \xra{\Gr_i(\upsilon)}
    \Gr_i(\mathcal{M})
    \xra{\Gr_i(\pi)}
    \Gr_i(\mathcal{M}/\mathcal{U})
    \ra
    0
  \end{equation*}
  for all $i$ (see \cite[1.1.11]{HodgeII}; this follows easily from the nine lemma).
  All underlying graded modules are in
  $\mathcal{B}$ (Remark \ref{rem:essim}). 
  Since $\Gr_i(\mathcal{M}) \cong \hat{L}_{x_i}$ is simple in
  $\mathcal{B}$, 
  either $\Gr_i(\upsilon)$ is an isomorphism and
  $\Gr_i(\mathcal{M}/\mathcal{U})=0$, or
  $\Gr_i(\mathcal{U})=0$ and $\Gr_i(\pi)$ is an isomorphism.
  We deduce that $\mathcal{U}$ and $\mathcal{M}/\mathcal{U}$
  are in $\dgFilMod$.
\end{proof}

\begin{proposition}
  \label{p:dgfilmod-equi}
  The obvious functor $\dgFilMod \ra \heart$ is an equivalence of
  categories. Any object in $\heart$ has finite length, and  
  the simple objects in $\heart$ are (up to isomorphism) the
  $\{\hat{L}_x\}_{x \in W}$.
\end{proposition}

\begin{proof}
  Let ${M}$, ${N}$ be in $\dgFilMod$. 
  Since both ${M}$ and ${N}$ are
  generated in degree zero, any homotopy $M \ra\{-1\}N$ is zero. So 
  the canonical map 
  \begin{equation*}
    \Hom_{\dgMod}({M},{N}) \ra \Hom_{\dgHot}({M},{N})
  \end{equation*}
  is an isomorphism, and the obvious functor $\dgFilMod \ra \dgFilt^0$
  is an equivalence. In combination with Proposition
  \ref{p:dgfiltzero-heart}, this shows the first statement.

  We prove the remaining statements for $\dgFilMod$.
  If ${M}\not= 0$ is in $\dgFilMod$, the first step
  of any $\hat{L}$-filtration yields a monomorphism $\hat{L}_y \hra
  {M}$. 
  If ${M}$ is simple, this must be an
  isomorphism. 
  Any non-zero subobject of $\hat{L}_x$ has a subobject isomorphic to
  some $\hat{L}_y$. So Lemma \ref{l:zero-or-iso} (or Remark
  \ref{rem:essim}) shows that  
  the $\{\hat{L}_x\}_{x \in W}$ are (up to isomorphism) the simple
  objects of $\dgFilMod$ and pairwise non-isomorphic. 
  Each object of $\dgFilMod$ has finite length, since it has an
  $\hat{L}$-filtration.
\end{proof}

\begin{remark}
  We could establish the existence of our t-structure on
  $\dgPrae$ and hence (cf. the proof of Theorem \ref{t:t-cat-prae}) on
  $\dgPer$ also  
  by starting from its potential heart, using \cite[1.3.13]{BBD}, as
  follows:
  \begin{enumerate}
  \item Define $\mathcal{C}$ to be the full subcategory of $\dgHot$
    whose objects are those isomorphic to an object of $\dgFilt^0$.
  \item Show that $\mathcal{C}$ is an admissible abelian subcategory
    of $\dgHot$ and closed under extensions.
    (Proposition \ref{p:dgfilmod-abelian-sub} equips $\dgFilt^0$ and
    hence $\mathcal{C}$ with the structure of an abelian category. 
    Short exact sequences in $\dgFilt^0$ are $A$-split and can be
    completed into distinguished triangles. Hence $\dgFilt^0$ and
    $\mathcal{C}$ are abelian admissible (using the octahedral axiom
    one can also conclude directly from the proof of Proposition
    \ref{p:dgfilmod-abelian-sub} that any morphism in $\mathcal{C}$ is
    $\mathcal{C}$-admissible). It is easy to see that any
    extension between objects of $\dgFilt^0$ is isomorphic to an
    object of $\dgFilt^0$. Hence $\mathcal{C}$ is stable by
    extensions.)
  \item The smallest strict full triangulated
    subcategory of $\dgHot$ containing
    all the $\{n\}\mathcal{C}$, for $n \in \DZ$, is $\dgPrae$ (easy).
  \end{enumerate}
  Then one can deduce
  Theorem \ref{t:filt-iso-prae}.
\remarkend
\end{remark}

\begin{proposition}
  \label{p:injective-in-dgheart}
  An object ${M}$ in the heart $\heart$ is injective if and only if
  $H^1({M})=0$.
\end{proposition}

\begin{proof}
  $M$ is injective if and only if $\Ext^1_{\heart}(N, {M})$ vanishes for all
  $N \in \heart$. By
  Proposition \ref{p:dgfilmod-equi} it is enough to check this for $N
  \in \{\hat L_x\}_{x \in W}$. Since $\mathcal{A}$ is isomorphic to
  a direct sum of the $\hat L_x$ with positive multiplicities, $M$ is
  injective if and only if 
  \begin{equation*}
    \Ext^1_{\heart}(\mathcal{A}, {M})=\Hom_{\dgHot}(\mathcal{A},
    \{1\}{M})= H^1({M})
  \end{equation*}
  vanishes (here the first equality follows from \cite[3.1.17.(ii)]{BBD}, the second
  one from \eqref{eq:Ahopro}).
\end{proof}

\begin{example}
  \label{ex:A-injective1}
  Assume that $H^1({\mathcal{A}})=0$ (since $d(A^0)=0$ this is
  equivalent to $Z^1(\mathcal{A})=0$). 
  Then $\mathcal{A}$ is an injective object of $\heart$, and any simple object $\hat L_x$ is
  injective (and projective). Hence any object of $\heart$ is a direct
  sum of simple objects.
\exampleend
\end{example}

If $R$ is an arbitrary ring, we denote by $\Mofover{R}$ the category
of finitely generated left $R$-modules.

\begin{proposition}
  \label{p:heart-module-category}
  Let $I$ be an object of the $\heart$. Assume that
  \begin{enumerate}
  \item $H^1(I)=0$ (equivalently: $I$ is injective in $\heart$),
  \item any simple object of the heart $\heart$ is contained in $I$,
    i.\,e.\ for all $x \in W$, there is a monomorphism $\hat L_x \ra I$ in
    $\heart$. (This is true, for example, if there is an inclusion
    $\mathcal{A} \subset I$ of dg modules.)
  \end{enumerate}
  Then the functor
  $\Hom_{\heart}(?,I): \heart \ra (\Modover{\End_{\heart}(I)})^\opp$
  induces an equivalence 
  \begin{equation*}
    \Hom_{\heart}(?,I): \heart \sira (\Mofover{\End_{\heart}(I)})^\opp.
  \end{equation*}
  between the heart $\heart$ and the opposite category of the category of
  finitely generated left $\End_{\heart}(I)$-modules.
\end{proposition}

\begin{proof}
  Every simple object of the artinian category $\heart$ is
  contained in the injective object $I$.
  Hence $I$ is a projective generator of the opposite
  category $\heart^\opp$, and a standard result yields an equivalence
  \begin{equation*}
    \Hom_{\heart^\opp}(I,?): \heart^\opp \sira \Mof(\End_{\heart^\opp}(I))=\Mofover{\End_{\heart}(I)}.
  \end{equation*}
  Now pass to the opposite categories.
\end{proof}

\begin{example}
  \label{ex:A-injective2}
  (Example~\ref{ex:A-injective1} continued.)
  Assume that $H^1({\mathcal{A}})=0$.
  We always have $\End_{\heart}(\mathcal{A})=H^0(\mathcal{A})=A^0$
  as rings (cf.\ \eqref{eq:Ahopro}).
  Hence we obtain an equivalence
  \begin{equation*}
    \Hom_{\heart}(?,\mathcal{A}): \heart \sira (\Mofover{A^0})^\opp,
  \end{equation*}
  i.\,e.\ the semisimple ring $A^0$ governs $\heart$. This example
  shows that in general the bounded derived category of the heart and
  $\dgPer(\mathcal{A})$ are not equivalent.
  \exampleend
\end{example}
\begin{example}
  \label{ex:A-truncated-poly}
  Assume that $k$ is a field and let $A=k[x]/x^{n+1}$ with $x$
  of degree $1$ and  $n \geq 2$. Let $\mathcal{A}=(A, d=0)$ and
  $M=A\oplus A$ as a graded $A$-module. Viewing elements of $M$ 
  as column vectors, we define a differential on $M$ by 
  $d_M= \tzmat 0x00$.
  This turns $M$ into a dg $\mathcal{A}$-module.
  (Note that, given any dg algebra $\mathcal{B}=(B, d=0)$ with vanishing
  differential, the differential of any (right) dg $\mathcal{B}$-module is a
  $B$-linear map (with square zero).)
  Here is our picture of $M$ for $n=2$ (note that the first summand in
  $M=A\oplus A$ corresponds to the lower row):
  \begin{equation*}
    \xymatrix@R=8pt@C=15pt{
      {\text{degree:}} & {-1} & {0} & {1} & {2} &
      {3} & {\dots} \\
      & {} & {k} \ar[rd]^{x} &
      {kx} \ar[rd]^{x} & {kx^2}
       & {} \\
      & {} & {k} &
      {kx} & {kx^2}
       & {} 
      }
  \end{equation*}
  We have $H(M)=k\oplus kx^n$, in particular $H^1(M)=0$, and $M$ is
  injective. Since $\mathcal{A} \subset M$, $a \mapsto \svek a0$ is a submodule, we obtain an equivalence
  \begin{equation*}
    \Hom_{\heart}(?,M): \heart \sira (\Mofover{\End_{\heart}(M)})^\opp.
  \end{equation*}
  Obivously $\End_{\heart}(M)=\big\{
  \tzmat {\lambda}{\mu}{0}{\lambda} \mid \lambda, \mu \in k\big\}\subset
  \Mat(2,k)$, which is isomorphic to $k[T]/T^2$ (for all $n \geq 2$).
\exampleend
\end{example}

\begin{remark}
\label{rem:socle-filt}
In example~\ref{ex:A-truncated-poly} the
endomorphism ring of the object $M$ consists of upper triangular
matrices. The reason for this is that morphisms respect the socle
filtration (in the example the socle filtration is $0 \subset A \oplus 0 \subset M$):

Since the heart $\heart$ is an artinian
category, any object $M$ has a largest semisimple subobject (an object
is semisimple if it is isomorphic to the direct sum of simple objects). 
This object is
called the socle of $M$ and denoted $\soc M$.
The socle filtration of $M$ is the unique increasing filtration
$(\soc_i M)_{i \geq 0}$ of $M$ such that $\soc_{i+1} M / \soc_i M = \soc(M / \soc_i
M)$ and $\soc_0 M = 0$. Morphisms respect the socle filtration. 

The socle is easy to describe for
objects $M \in \dgFilMod$: Let $Z^0(M)$ be the $0$-cycles in $M$, i.\,e.\ the kernel of
$d_M|_{M^0}:M^0 \ra M^1$. Obviously, multiplication 
$Z^0(M) \otimes_{\mathcal{A}^0} \mathcal{A} \ra M$,
$z \otimes a \mapsto za$, defines a monomorphism of dg modules. We
claim that its image $Z^0(M)A \subset M$ is the socle of $M$:
\begin{equation*}
  Z^0(M)A =\soc M.
\end{equation*}
Since $Z^0(M)A$ is semisimple, we have $Z^0(M)A\subset \soc M$. On the
other hand, if $\hat L_x \subset \soc M$, then $\hat L_x^0 \subset Z^0(M)$
(since $d(\hat L_x^0)=0$), and hence $\hat L_x=\hat L_x^0 A \subset Z^0(M)A$. This
implies $\soc M \subset Z^0(M)A$.

Since $M/\soc M$ is in $\dgFilMod$ again, this yields an inductive
description of the socle filtration of $M$.

We also see that the subquotients of the socle filtration of $M \in
\dgFilMod$ are projective as graded $A$-modules. This shows that $M$
has a direct sum decomposition
\begin{equation*}
  M = M_1 \oplus M_2 \oplus \dots \oplus M_m 
\end{equation*}
as a graded $A$-module 
such that 
\begin{equation*}
  \soc_i M = M_1 \oplus \dots \oplus M_i
\end{equation*}
as dg modules (for $1 \leq i \leq m$).

Let $N$ be another object of $\dgFilMod$ and choose a similar
decomposition $N =N_1 \oplus \dots \oplus N_n$ with $\soc_j N= N_1
\oplus \dots \oplus N_j$.
Since morphisms of dg $\mathcal{A}$-modules are in particular morphisms of graded
$A$-modules we have an inclusion 
\begin{align*}
  \Hom_{\heart}(M,N)=\Hom_{\dgMod(\mathcal{A})}(M,N) & \subset
  \Hom_{\gMod(A)}(M,N)\\ 
  & =\bigoplus_{i,j}\Hom_{\gMod(A)}(M_i,N_j).
\end{align*}
If we view elements of $\Hom_{\gMod(A)}(M,N)$ as matrices, any element
of $\Hom_{\heart}(M,N)$ is upper triangular since it respects the
socle filtration.
\remarkend
\end{remark}

\section{Homotopically Minimal DG Modules}
\label{sec:homot-minim-dg}

We prove that the objects of $\dgFilt(\mathcal{A})$ are precisely the
so called 
homotopically minimal
objects in $\dgPer(\mathcal{A})$ and give alternative
characterizations of these objects. We continue to assume that $\mathcal{A}$ satisfies
\ref{enum:pg}-\ref{enum:sdga}. 

We say that a dg module $M$ is \define{homotopically minimal} 
(cf.\ \cite{AFH}, and
\cite{avramov-martsinkovsky}
or \cite[App.\ B]{krause-stable-derived} for complexes)
if any endomorphism $f:M \ra M$
in $\dgMod(\mathcal{A})$ that is a homotopy equivalence (i.\,e.\ it
becomes an isomorphism in $\dgHot(\mathcal{A})$) is an isomorphism.

\begin{remark}
In \cite[11.4]{BL} (where $\mathcal{A}$ is a polynomial ring with generators in
positive even degrees) the term ``minimal $\mathcal{K}$-projective'' is
defined and it is proved that the minimal $\mathcal{K}$-projective
modules are homotopically minimal in our sense. The equivalent statements (in
particular \ref{enum:dgper-lengthminimal}) of the next proposition
show that the homotopically minimal dg modules in $\dgPer$ are ``minimal
$\mathcal{K}$-projective''.
\remarkend
\end{remark}

\begin{proposition}
  \label{p:dgper-tfae}
  Let $\mathcal{M}=(M,d)$ be in $\satzdgPer(\mathcal{A})$. Then the
  following statements are equivalent.
  \begin{enumerate}
  \item\label{enum:dgper-dgfilt} $\mathcal{M}$ is in $\satzdgFilt$;
  \item\label{enum:dgper-minimal} $\mathcal{M}$ is homotopically minimal;
  \item\label{enum:dgper-leq} for all $i$, $M_{\leq i}$ is an $\mathcal{A}$-submodule of
    $\mathcal{M}$, that is $d(M_{\leq i})\subset M_{\leq i}$;
  \item\label{enum:dgper-dplus} $d(M) \subset MA^+$;
  \item\label{enum:dgper-dzero} $\ol{\mathcal{M}}= \mathcal{M} \otimes_{\mathcal{A}}
    \mathcal{A}^0= \mathcal{M}/\mathcal{M}\mathcal{A}^+$ has differential
    $d=0$;
  \item\label{enum:dgper-length}
    $\length(\ol{{M}})=\length(H(\ol{\mathcal{M}}))$;
  \item\label{enum:dgper-lengthminimal} 
    $\lambda(\ol{M})$
    is finite and minimal in the
    isomorphism class of $\mathcal{M}$ in $\dgHot$.
  \end{enumerate}
\end{proposition}

\begin{proof}
  \begin{itemize}
  \item \ref{enum:dgper-dgfilt} $\Rightarrow$ \ref{enum:dgper-minimal}:
    Let ${\mathcal{M}}$ be in $\dgFilt$ and $f:{\mathcal{M}} \ra
    {\mathcal{M}}$ a morphism in $\dgMod(\mathcal{A})$ that 
    becomes an isomorphism in $\dgHot(\mathcal{A})$. Then 
    $\ol f: \ol{{\mathcal{M}}} \ra \ol{{\mathcal{M}}}$ in $\dgMod(\mathcal{A}^0)$ becomes an
    isomorphism in $\dgHot(\mathcal{A}^0)$. Since the differential of
    $\ol {{\mathcal{M}}}$ vanishes, $\ol f$ is already an
    isomorphism. Since $M$ is in $\gProjf(A)$ 
    (Lemma \ref{l:inclusions}) 
    and in particularly
    $A$-flat, Lemma \ref{l:isocrit} shows that $f$ is an isomorphism.
  \item \ref{enum:dgper-minimal} $\Rightarrow$ \ref{enum:dgper-dgfilt}:
    Let ${\mathcal{M}}$ be homotopically minimal.
    By Theorems \ref{t:filt-iso-prae}
    and \ref{t:t-cat-prae} there is an object ${\mathcal{F}}$ in $\dgFilt$
    isomorphic to ${\mathcal{M}}$ in $\dgHot$. We know already that
    ${\mathcal{F}}$ is homotopically minimal, hence ${\mathcal{M}}$ and ${\mathcal{F}}$
    are already isomorphic in $\dgMod$.
  \item \ref{enum:dgper-dgfilt} $\Rightarrow$ \ref{enum:dgper-leq}:
    Each $M_{\leq i}$ appears in any $\hat L$-filtration.
  \item \ref{enum:dgper-leq} $\Rightarrow$
    \ref{enum:dgper-dplus}:
    Let $m$ be in $M^i$. Then $d(m) \in M_{\leq i}$ by assumption.
    Having degree $i+1$ it must lie in $M_{\leq i}A^+$.
  \item \ref{enum:dgper-dplus} $\Rightarrow$
    \ref{enum:dgper-leq}:
    Since $d(M^j A^l)\subset d(M^j) A^l + M^j A^{l+1}$ by the Leibniz
    rule, it is
    sufficient to prove that $d(M^j) \subset M_{\leq j}$.  
    But $d(M^j) \subset M^{j+1}\cap MA^+=M^jA^1+M^{j-1}A^2+\dots
    \subset M_{\leq j}$.
  \item \ref{enum:dgper-dplus} $\Leftrightarrow$
    \ref{enum:dgper-dzero} $\Rightarrow$ \ref{enum:dgper-length}: Obvious.
  \item 
    \ref{enum:dgper-length}
    $\Rightarrow$ 
    \ref{enum:dgper-dzero}:
    It is obvious that $\lambda(H(\ol{\mathcal{N}}))$ is finite for any $\mathcal{N}$ in
    $\dgPer(\mathcal{A})$ since this is true for $\mathcal{A}$ itself.
    Let $Z \subset \ol{\mathcal{M}}$ be the kernel of the differential
      of $\ol{\mathcal{M}}$ and $B$ its image. Then
      $H(\ol{\mathcal{M}})=Z/B$ and
      \begin{equation*}
        \lambda(H(\ol{\mathcal{M}}) \leq \lambda(Z) \leq \lambda(\ol{M})
      \end{equation*}
      with equalities if and only if $B=0$ and $Z=\ol{\mathcal{M}}$.
  \item \ref{enum:dgper-length} $\Rightarrow$
    \ref{enum:dgper-lengthminimal}:
    As already remarked, $\lambda(H(\ol{\mathcal{N}}))$ is finite for any $\mathcal{N}$ in
    $\dgPer(\mathcal{A})$.
    Now let $\mathcal{N}$ be isomorphic to $\mathcal{M}$ in $\dgHot$. Then
    \begin{equation*}
      \length (\ol{{M}})=\length (H(\ol{{\mathcal{M}}})) =\length
      (H(\ol{{\mathcal{N}}})) \leq \length(\ol{{N}}).
    \end{equation*}
  \item \ref{enum:dgper-lengthminimal} $\Rightarrow$ \ref{enum:dgper-dgfilt}:
    By Theorems \ref{t:filt-iso-prae} and \ref{t:t-cat-prae} there is
    an object $\mathcal{F}=(F,d)$ in $\dgFilt$ and a morphism
    $f:\mathcal{M} \ra \mathcal{F}$ in $\dgMod(\mathcal{A})$ that
    becomes an isomorphism in $\dgHot(\mathcal{A})$.
    Then $\ol f: \ol{\mathcal{M}} \ra \ol{\mathcal{F}}$ in $\dgMod(\mathcal{A}^0)$ becomes
    an isomorphism in $\dgHot(\mathcal{A}^0)$. 
    We have already proved \ref{enum:dgper-dgfilt} $\Rightarrow$
    \ref{enum:dgper-length}, hence
    \begin{equation*}
      \length (\ol{{F}})=\length (H(\ol{\mathcal{F}})) =\length
      (H(\ol{\mathcal{M}})) \leq \length(\ol{{M}}).
    \end{equation*}
    Since $\lambda(\ol{M})$ is minimal in the isomorphism class of
    $\mathcal{M}$ we have equality. Hence the differential of
    $\ol{\mathcal{M}}$ vanishes thanks to \ref{enum:dgper-dzero} 
    $\Leftrightarrow$ \ref{enum:dgper-length}. 
    The same is true for $\ol{\mathcal{F}}$.
    Thus $\ol{f}$ is an isomorphism in $\dgMod(\mathcal{A}^0)$. Since
    the underlying graded module of $\mathcal{F}$ is in $\gProjf(A)$ 
    (Lemma \ref{l:inclusions}) and in
    particular $A$-flat, Lemma \ref{l:isocrit} shows that $f$ is an isomorphism.
  \end{itemize}
\end{proof}

\begin{corollary}
  \label{c:dgfilt-closed-summands}
  The category $\satzdgFilt$ is closed under forming direct summands in
  $\dgMod(\mathcal{A})$. 
\end{corollary}

\begin{proof}
  Let $F$ be in $\dgFilt$ and $F \cong X \oplus Y$ a decomposition
  in $\dgMod(\mathcal{A})$. Then $X$ and $Y$ are in
  $\dgPer(\mathcal{A})$, and in $\dgFilt$ by
  \ref{enum:dgper-dgfilt} $\Leftrightarrow$ \ref{enum:dgper-leq}.
\end{proof}

\section{Indecomposables}
\label{sec:indecomposables}

We prove that $\dgPer(\mathcal{A})$ is a Krull-Remak-Schmidt category,
i.\,e.\ it is an additive category and every object decomposes into a finite direct sum of objects
having local endomorphism rings.
This definition is taken from
\cite[App.~A]{Chen-Ye-Zhang-algebras-of-derived-dimension-zero-KS-category}
(cf.~\cite{auslander-I-II,Ringel-Tame-algebras-integral-forms,Krause-krull-schmidt-categories}); 
it is shown there that such a decomposition is essentially unique.
The dg algebra $\mathcal{A}$ is as before.

\begin{proposition}
  \label{p:bijection-isoclasses}
  The inclusion $\satzdgFilt \subset \satzdgPer$ induces a bijection
  between 
  \begin{enumerate}
  \item\label{enum:objdgfilt} objects of $\satzdgFilt$ up to isomorphism in $\dgMod$ and
  \item\label{enum:objdgper} objects of $\satzdgPer$ up to isomorphism in $\dgHot$.
  \end{enumerate}
\end{proposition}

\begin{proof}
  The map from the set in \ref{enum:objdgfilt} to the set in
  \ref{enum:objdgper} is surjective since $\dgFilt \subset \dgPer$ is
  an equivalence (see Theorems \ref{t:filt-iso-prae} and
  \ref{t:t-cat-prae}).
  If $M$ and $N$ are in $\dgFilt$ and isomorphic in $\dgHot$ then they
  are already isomorphic in $\dgMod$, since $M$ and $N$ are homotopically minimal
  (Proposition \ref{p:dgper-tfae}). This proves injectivity.
\end{proof}

\begin{corollary}
  \label{c:indecomposability-for-dgfilt}
  An object $M$ of $\satzdgFilt$ is indecomposable in $\dgMod$ 
  if and only if it is indecomposable in $\dgHot$ (in $\satzdgFilt$, in
  $\satzdgPer$, respectively).
\end{corollary}

\begin{proof}
  Assume that $M$ is indecomposable in $\dgMod$ and let $M \cong
  A\oplus B$ be a direct sum decomposition in $\dgHot$ (in $\dgFilt$,
  in $\dgPer$).
  Then $A$, $B \in \dgPer$, and Proposition 
  \ref{p:bijection-isoclasses} shows that we can assume that $A$,
  $B$ are in $\dgFilt$; but then $M$ and $A\oplus B$ are objects of
  $\dgFilt$ that become isomorphic in $\dgHot$. Hence they are already
  isomorphic in $\dgMod$. This implies that $A=0$ or $B=0$, so $M$ is indecomposable
  in $\dgHot$ (in $\dgFilt$,
  in $\dgPer$).

  Now assume that $M$ is indecomposable in $\dgHot$ (in $\dgFilt$, in $\dgPer$). 
  Let $M\cong A\oplus B$ be a
  direct sum decomposition in $\dgMod$. 
  Corollary \ref{c:dgfilt-closed-summands} shows that $A$ and $B$
  are in $\dgFilt$. Since $M$ is indecomposable in $\dgHot$ (in
  $\dgFilt$, $\dgPer$) we must
  have $A\cong 0$ or $B \cong 0$ in $\dgHot$, and by Proposition
  \ref{p:bijection-isoclasses} we obtain $A=0$ or $B=0$ in
  $\dgMod$. So $M$ is indecomposable in $\dgMod$.
\end{proof}

We need the following easy lemma in order to obtain a Fitting
decomposition for objects of $\dgFilt$.
The idea is to get a decomposition of the middle term of a short exact
sequence from decompositions of the border terms.
\begin{lemma}
  \label{l:fitting-ses}
  Let $X: 0 \ra A \ra B \ra C \ra 0$ be an exact sequence in an abelian
  category and $f=(a,b,c)$ an endomorphism of $X$.
  Assume that
  $\Bild a  = \Bild a^2$ and
  $\Kern c  = \Kern c^2$.
  Then the sequences 
  \begin{align}
    \label{eq:kernfolge}
    0 \ra \Kern a^2 \ra & \Kern b^2 \ra \Kern c^2 \ra 0 
    \\
    \label{eq:bildfolge}
    0 \ra \Bild a^2 \ra & \Bild b^2 \ra \Bild c^2 \ra 0 
  \end{align}
  are exact.
  If in addition
  \begin{align*}
    \Kern a  = \Kern a^2, & \quad A = \Kern a \oplus \Bild a,\\
    \Bild c  = \Bild c^2, & \quad C = \Kern c \oplus \Bild c,
  \end{align*}
  then
  \begin{equation*}
    \Kern b^2  = \Kern b^3, \quad \Bild b^2  = \Bild b^3, \quad B = \Kern b^2 \oplus \Bild b^2.
  \end{equation*}
\end{lemma}

\begin{proof}
  We view $X$ as a complex concentrated in degrees $0, 1, 2$, with $A$
  sitting in degree $0$.
  Consider the short exact sequences $\Sigma_1: 0\ra \Kern f \ra X \ra \Bild f \ra 0$ and 
  $\Sigma_2: 0 \ra \Kern f^2 \ra X \ra \Bild f^2 \ra 0$ of
  complexes and note that there is a morphism  $\sigma: (f, f, i): \Sigma_2 \ra
  \Sigma_1$, where $i$ is the inclusion $\Bild f^2 \ra \Bild f$. 

  The long exact cohomology sequence obtained from $\Sigma_2$ shows
  the exactness of \eqref{eq:kernfolge} and \eqref{eq:bildfolge} at all positions
  except at $\Kern c^2$ and $\Bild b^2$.
  The non-trivial part of the morphism between the long exact
  sequences induced by $\sigma$ is the commutative square
  \begin{equation*}
    \xymatrix{
      {H^1(\Bild f^2)} \ar[r]^\sim \ar[d]^{H^1(i)} & {H^2(\Kern f^2)} \ar[d]^{H^2(f)}\\
      {H^1(\Bild f)} \ar[r]^\sim & {H^2(\Kern f).} 
    }
  \end{equation*}
  Our assumption $\Kern c = \Kern c^2$ shows that the morphism
  $H^2(f)$ is zero. From $\Bild a=\Bild a^2$ we see
  that $H^1(i)$ is a monomorphism. 
  This implies that the upper corners in our diagram are zero. 
  Hence \eqref{eq:kernfolge} and \eqref{eq:bildfolge} are also exact
  at $\Kern c^2$ and $\Bild b^2$.

  Assume that the additional assumptions hold. Then 
  the five-lemma shows that the obvious morphism from the direct sum
  of the two exact sequences \eqref{eq:kernfolge} and
  \eqref{eq:bildfolge} to the exact sequence $X$ is an
  isomorphism. Thus $B =\Kern b^2 \oplus \Bild b^2$. 

  Since $\Bild a^2=\Bild a^4$ and $\Kern c^2 = \Kern c^4$ we obtain
  from the first part an exact sequence of the form
  \eqref{eq:kernfolge} with squares replaced by fourth powers, and a
  morphism from 
  \eqref{eq:kernfolge} to this sequence. The five lemma again shows
  that $\Kern b^2=\Kern b^4$, hence $\Kern b^2=\Kern b^3$.
  Similarly one proves $\Bild b^2 = \Bild b^3$.
\end{proof}

\begin{proposition}[Fitting decomposition]
  \label{p:fitting-dgFilt}
  Let $M$ be in $\satzdgFilt(\mathcal{A})$ and $f: M \ra M$ an endomorphism in
  $\dgMod$.
  Then  $\Kern f^n=\Kern f^{n+1}$, $\Bild f^n= \Bild f^{n+1}$ and
  $M = \Kern f^n \oplus \Bild f^n$ for $n \gg 0$.
\end{proposition}

\begin{proof}
  Each object $M$ of $\dgFilt$ lies in $\dgFilt^I$ for some bounded interval
  $I=[a,b]$ where $a \leq b$ are integers. We prove the statement by
  induction over $|I|:=b-a$.
  For $|I|=0$ we can use the classical Fitting lemma since the heart
  of our t-structure is artinian.

  Assume that $|I|>0$. Choose $i \in I$ such that 
  $0 \subsetneq M_{\leq i} \subsetneq M$. Then we have a short exact
  sequence 
  \begin{equation*}
    0 \ra M_{\leq i} \ra M \ra M/M_{\leq i} \ra 0
  \end{equation*}
  in $\dgMod$,
  and $f$ induces an endomorphism
  of this sequence. 
  By induction we have Fitting decompositions for $M_{\leq i}$ and
  $M/M_{\leq i}$, and Lemma \ref{l:fitting-ses} shows that these
  decompositions yield a Fitting decomposition of $M$.
\end{proof}

\begin{proposition}
  \label{p:krs-category}
  \begin{enumerate}
  \item 
    \label{enum:krs-endos} 
    Let $M$ be an indecomposable object of $\satzdgFilt$ (cf. Corollary
    \ref{c:indecomposability-for-dgfilt}). Then any
    element of $\End_{\dgMod}(M)$ is either an isomorphism or
    nilpotent, and the nilpotent endomorphisms form an ideal in
    $\End_{\dgMod}(M)$.
    The same statement holds for $\End_{\dgHot}(N)$ if $N$ is an
    indecomposable object of $\satzdgPer$.
    In particular both endomorphism rings are local rings.
  \item 
    \label{enum:krs-decomp} 
    Each object of $\satzdgFilt$ decomposes into a finite direct sum of
    indecomposables of $\satzdgFilt$. Same for $\satzdgPer$. 
  \item 
    \label{enum:krs-krs} 
    $\satzdgFilt$ and $\satzdgPer$ are Krull-Remak-Schmidt categories.
  \end{enumerate}
\end{proposition}

\begin{proof}
  The first statement of \ref{enum:krs-endos} is a standard consequence of the Fitting
  decomposition. The second statement then follows from Proposition
  \ref{p:bijection-isoclasses}
  and Corollary \ref{c:indecomposability-for-dgfilt}.

  Let us prove \ref{enum:krs-decomp}.
  If $X$ is an object of $\dgFilt$ with a direct sum decomposition $X
  \cong Y\oplus Z$ in $\dgMod$, then $Y$ and $Z$
  are in $\dgFilt$ by Corollary \ref{c:dgfilt-closed-summands}, and
  $\lambda(\ol X)= \lambda(\ol Y) + \lambda(\ol Z)$.
  Moreover $\lambda(\ol X)$ is finite, and $\lambda(\ol X)=0$ if and
  only if $X=0$.

  This shows that any object $X$ of $\dgFilt$ has a finite direct sum decomposition
  $X \cong M_1 \oplus \dots \oplus M_n$ in $\dgMod$ where all $M_i$
  are in $\dgFilt$ and indecomposable in $\dgMod$.
  All $M_i$ are also indecomposable in $\dgFilt$ (Corollary
  \ref{c:indecomposability-for-dgfilt}).

  The analog statement for $\dgPer$ is true thanks to the equivalence
  $\dgFilt \subset \dgPer$ (Theorems \ref{t:filt-iso-prae} and
  \ref{t:t-cat-prae}).

  Part \ref{enum:krs-krs} is a consequence of 
  \ref{enum:krs-endos} and \ref{enum:krs-decomp}.
\end{proof}

\section{Koszul-Duality}
\label{sec:koszul-duality}
We study the case that the dg algebra $\mathcal{A}$ is a Koszul ring
with differential zero.
Under some finiteness conditions we show that the heart of the 
t-structure on $\dgPer(\mathcal{A})$ is governed by the dual Koszul
ring.
This is a shadow of the usual Koszul equivalence (see
Remark~\ref{rem:koszul-shadow} below).
We assume that $k=\DZ$ in this chapter.

Let $A$ be a Koszul ring (see \cite{BGS}). It gives rise to a dg
algebra $\mathcal{A}=(A, d=0)$ with differential zero.
This dg algebra satisfies the conditions \ref{enum:pg}-\ref{enum:sdga}.
Since $A$ is Koszul there is a
resolution $P \sra A^0$,
\begin{equation*}
  \dots \ra P_{-2} \xra{d_{-2}} P_{-1} \xra{d_{-1}} P_0 \sra A^0,
\end{equation*}
of the graded (right) $A$-module $A^0=A/A^+$ such that each $P_{-i}$ is a
projective object in $\gMod(A)$ and 
generated in
degree $i$, $P_{-i}=P_{-i}^i A$.
Such a resolution is unique up to unique isomorphism.
Note that multiplication defines an
isomorphism
\begin{equation}
  \label{eq:mult-iso}
  P_{-i}^i \otimes_{A^0} A \sira P_{-i}
\end{equation}
(it is surjective, splits by the projectivity of $P_{-i}$,
and any splitting is surjective, since it is an 
isomorphism in degree $i$).

Since the differential of $\mathcal{A}$ vanishes, we can view each
$P_{-i}$ as a dg $\mathcal{A}$-module with differential zero. The maps
$d_{-i}$ are then morphisms of dg modules.

On the graded $A$-module
\begin{equation*}
  K:= K(A):=\bigoplus_{n \in \DN} \{n\}P_{-n} = P_0 \oplus \{1\}P_{-1} \oplus \dots 
\end{equation*}
we define a differential $d_K$ by
\begin{multline}
  \label{eq:diffkoszul}
  d_K(p_0,p_{-1},p_{-2},\dots) \\ 
  := ((-1)^{-1}d_{-1}(p_{-1}),
  (-1)^{-2}d_{-2}(p_{-2}), (-1)^{-3}d_{-3}(p_{-3}),\dots).
\end{multline}
This defines an object $\mathcal{K}(A)=(K,d_K)$ of
$\dgMod(\mathcal{A})$. 
It can be seen as an iterated mapping cone of the morphisms
$(-1)^{-i}d_{-i}$ of dg modules (if there are only finitely many summands). 
(We chose the signs from the proof
of Theorem \ref{t:endo-koszul}.
If we omit all the signs $(-1)^{-i}$ in \eqref{eq:diffkoszul}, we obtain an isomorphic dg module.) 

We say that $P$ is of finite length if $P_{-j}=0$
for $j \gg 0$, and that $P$ has finitely generated components
if each $P_{-i}$ is a finitely generated $A$-module.
These two conditions mean precisely that $K$ is a finitely generated
$A$-module.

\begin{proposition}
  \label{p:koszul-dg-mod}
  Let $A$ be a Koszul ring, $\mathcal{A}=(A, d=0)$, and
  $P \sra A^0$ a resolution as above. Assume that $P$ is of finite
  length with finitely generated components.
  Then $\mathcal{K}(A)$ is in $\satzdgFilt^0$
  and an injective object of the heart $\heart$.
\end{proposition}

\begin{proof}
  Since $P$ has finite length, the filtration $0 \subset P_0 \subset
  P_0 \oplus \{1\}P_{-1} \subset \dots \subset K$ by
  $\mathcal{A}$-submodules stabilizes after finitely many steps. 
  Since all subquotients $\{i\}P_{-i}$ have differential zero and are
  finitely generated by their 
  degree zero part, they have an $\hat L$-filtration (use
  \eqref{eq:mult-iso}).
  Hence $\mathcal{K}(A)$ is in $\dgFilt^0 \subset \heart$.

  Since $P \sra A^0$ is a resolution, the cohomology
  $H(\mathcal{K}(A))$ is $A^0$, sitting in degree zero. Hence
  $\mathcal{K}(A)$ is an injective object of $\heart$
  by Proposition \ref{p:injective-in-dgheart}. 
\end{proof}

For a positively graded ring $A$ form the graded ring 
$E(A):=\Ext^\bullet_{\Mod(A)}(A^0,A^0)$ 
of self-extensions of the right
$A$-module $A^0=A/A^+$ (we consider self-extensions in $\Mod(A)$, not
in $\gMod(A)$). This ring is called the Koszul dual ring of $A$ (this
definition is slightly different from the definition in \cite{BGS}:
they consider self-extensions of the left $A$-module $A^0$).

\begin{theorem}
  \label{t:endo-koszul}
  Let $A$ be a Koszul ring, $\mathcal{A}=(A, d=0)$, and
  $P \sra A^0$ a resolution as above. Assume that $P$ is of finite
  length with finitely generated components.
  Then:
  \begin{enumerate}
  \item\label{en:kk-equi}  
    The functor $\DK:=\Hom_{\heart}(?,\mathcal{K}(A)):
    \heart \ra
    (\Modover{\End_{\heart}(\mathcal{K}(A))})^\opp$
    induces an equivalence 
    \begin{equation*}
      \DK:\heart \sira (\Mofover{\End_{\heart}(\mathcal{K}(A))})^\opp
    \end{equation*}
    between $\heart$ and the opposite category of the category of
    finitely generated (left) $\End_{\heart}(\mathcal{K}(A))$-modules.
  \item\label{en:endo-kdg} 
    $\End_{\heart}(\mathcal{K}(A))$ is
    isomorphic to $E(A)=\Ext^\bullet_{\Mod(A)}(A^0,A^0)$.
  \end{enumerate}
\end{theorem}

\begin{example}
  \label{ex:koszul-and-bgs}
  Let $A=\DC[X]$ be a polynomial ring with $X$ homogeneous of degree
  one. This is a Koszul ring with Koszul dual ring the ring of dual numbers
  $B:=E(A)=\DC[\varepsilon]/(\varepsilon^2)$.
  The resolution $P \sra A^0$ is given by $\dots \ra 0 \ra \{-1\}A
  \xra{X} A \sra \DC$, and $\mathcal{K}(A)$ is $A \oplus A$
  as a graded $A$-module with differential $\tzmat 0{-X}00$. 
  Theorem~\ref{t:endo-koszul} says that
  $\End_{\heart}(\mathcal{K}(A))$ is isomorphic to $B$ (which is of
  course easy to check, cf.\ Example \ref{ex:A-truncated-poly}), and
  that there is an equivalence 
  $\heart \sira (\Mofover{B})^\opp$.

  More generally, we could take $A=SV$ the symmetric algebra of some
  finite dimensional vector space $V$ (with all elements of $V$ in
  degree one). Then $E(A)=\bigwedge V^*$ is
  the exterior algebra on the dual space of $V$, and we obtain an
  equivalence
  $\heart \sira (\Mofover{\bigwedge V^*})^\opp$.
  \exampleend  
\end{example}

\begin{proof}
  Part \ref{en:kk-equi} follows from 
  Proposition~\ref{p:koszul-dg-mod}, the fact that 
  $\mathcal{A}=P_0$ is an $\mathcal{A}$-submodule of
  $\mathcal{K}(A)$ and Proposition~\ref{p:heart-module-category}.

For the proof of \ref{en:endo-kdg} we need the following general
construction. 
Let $S=(S, d^S)=(\dots \ra S_i \xra{d_i} S_{i+1} \ra \dots)$ 
and $T=(T, d^T)$
be complexes in some abelian category. Then
we define a complex of abelian groups $\complexHom(S,T)$ as follows:
Its $i$-th component is
\begin{equation*}
  \complexHom^i(S, T)
  =\prod_{s+t=i}\Hom(S_{-s}, T_t),
\end{equation*}
and its differential is given by 
\begin{equation*}
  df = d^T \comp f - (-1)^i f \comp d^S
\end{equation*}
for $f$ homogeneous of degree $i$. 
Note that $\complexEnd(S):=\complexHom(S,S)$ with the obvious composition becomes a dg
algebra.

  We now prove part \ref{en:endo-kdg}:
  Let $v: \gMod(A) \ra \Mod(A)$ be the 
  ``forgetting the grading''
  functor. 
  We denote the induced functor from the category of complexes in
  $\gMod(A)$ to that of complexes in $\Mod(A)$ by the same letter.
  
  If $M$ and $N$ are in $\gMod(A)$, the map
  \begin{align}
    \label{eq:hom-grading}
    \bigoplus_{n \in \DZ} \Hom_{\gMod(A)}(M, \{n\}N) & \ra \Hom_{\Mod(A)}(v(M), v(N))\\
    f=(f_n) & \mapsto \sum f_n \notag
  \end{align}
  is always injective. It is bijective if $M$ is finitely generated as
  an $A$-module. 

  Since $v(P) \sra v(A_0)$ is a projective resolution in $\Mod(A)$,
  the cohomology of $\mathcal{R}:=\complexEnd(v(P))$ is $E(A)$.
  The $i$-th component of $\mathcal{R}$ is  
  \begin{align*}
    \mathcal{R}^i
    & =
    \prod_{s \in \DZ}\Hom_{\Mod(A)}(v(P_{-s}), v(P_{-s+i})) \\
    & =
    \bigoplus_{s \in \DZ}\Hom_{\Mod(A)}(v(P_{-s}), v(P_{-s+i})) & &
    \text{$v(P)$ has finite length,}\\
    & \sila
    \bigoplus_{s, n \in \DZ} \Hom_{\gMod(A)}(P_{-s}, \{n\} P_{-s+i}) & &
    \text{$P_{-s}$ finitely gen., \eqref{eq:hom-grading},}\\
    & =
    \bigoplus_{n \in \DZ} \prod_{s \in \DZ}\Hom_{\gMod(A)}(P_{-s}, \{n\} P_{-s+i}) & &
    \text{$P$ has finite length,}\\
    & =
    \bigoplus_{n \in \DZ} \complexHom^i_{\gMod(A)}(P, \{n\} P). & &
  \end{align*}
  So if we define $\mathcal{R}_n = \complexHom_{\gMod(A)}(P,
  \{n\}P)$ for $n \in \DZ$,
  we get an isomorphism of complexes
  \begin{equation*}
    \mathcal{R} \sila \bigoplus_{n \in \DZ} \mathcal{R}_n.
  \end{equation*}
  Note that $H^i(\mathcal{R}_n) = \Ext^i_{\gMod(A)}(A_0, \{n\} A_0)$ vanishes if $i \not= -n$ 
  (Proof (cf.~\cite[Prop.~2.1.3]{BGS}): Even the complex $\complexHom_{\gMod(A)}(P, \{n\}A_0)$ vanishes in
  all degrees $\not= -n$). 
  Hence we obtain
  \begin{equation*}
    E^i(A) =H^i(\mathcal{R})\sila \bigoplus_{n\in \DZ}
    H^i(\mathcal{R}_n) = H^i(\mathcal{R}_{-i}).
  \end{equation*}
  In order to compute $H^i(\mathcal{R}_{-i})$, we observe that
  \begin{align*}
    \mathcal{R}_{-i}^j
    & = \prod_{s \in \DZ}\Hom_{\gMod(A)}(P_{-s}, \{{-i}\}P_{-s+j}) \\
    & = \bigoplus_{s \in \DZ}\Hom_{\gMod(A)}(P_{-s}, \{{-i}\}P_{-s+j}) \\
  \end{align*}
  vanishes if $j < i$, since $P_{-s}=P_{-s}^sA$ is generated in degree
  $s$ and
  $(\{{-i}\}P_{-s+j})^s=P^{s-i}_{-s+j}$ vanishes for $s-i< s-j$.
  This implies 
  \begin{equation*}
    E^i(A)
    =H^i(\mathcal{R}_{-i})
    =\Kern
    (\mathcal{R}_{-i}^i \ra \mathcal{R}_{-i}^{i+1}),
  \end{equation*}
  so $E(A)=\bigoplus_{i \in \DZ} E^i(A)$ is a subset of
  $\bigoplus_{i \in \DZ} \mathcal{R}_{-i}^i$. 
  On the other hand we have
  \begin{align*}
    \End_{\heart}(\mathcal{K}(A)) 
    & = \End_{\dgMod(\mathcal{A})}(\mathcal{K}(A)) \\
    & \subset \End_{\gMod(A)}(K(A)) \\
    & = \bigoplus_{i,s \in \DZ}
    \Hom_{\gMod(A)}(\{s\}P_{-s}, \{{s-i}\}P_{-s+i}) \\
    & =\bigoplus_{i \in \DZ}\mathcal{R}^i_{-i}.
  \end{align*}
  In order to identify these two subsets of $\bigoplus_{i \in
    \DZ}\mathcal{R}^i_{-i}$, let
  $f=(f^i)_{i \in \DZ}=(f^i_s)_{i, s \in \DZ}$ be an element of
  $\bigoplus_{i \in \DZ} \mathcal{R}_{-i}^i$ with 
  \begin{equation*}
    f^i_s \in 
    \Hom_{\gMod(A)}(\{s\}P_{-s}, \{{s-i}\}P_{-s+i}).
  \end{equation*}
  Then it is easy to check that
  $f \in E(A)$ if and only if
  \begin{equation*}
    d_{-s+i} \comp f_s^i - (-1)^i f_{s-1}^i \comp d_{-s} = 0 \quad \text{for all $i$, $s \in \DZ$,}
  \end{equation*}
  if and only if $f \in \End_{\heart}(\mathcal{K}(A))$.
\end{proof}

\begin{remark}
  Since the graded $A$-module $K(A)$ is a (finite) direct sum,
  its endomorphism ring $\End_{\gMod(A)}(K(A))$ consists of matrices and can be equipped with
  a ``diagonal'' $\DZ$-grading such that the piece 
  \begin{equation*}
    \Hom_{\gMod(A)}(\{s\}P_{-s},\{t\}P_{-t})
  \end{equation*}
  has degree $s-t$. 
  These matrices are not upper triangular in general.
  The particular form of $\mathcal{K}(A)$ implies that $\End_{\heart}(\mathcal{K}(A)) \subset
  \End_{\gMod(A)}(K(A))$ 
  is a graded subalgebra.
  In fact it consists of upper triangular matrices, i.\,e.\ it is
  positively (= non-negatively) graded:
  This can be directly deduced from the proof of 
  Theorem~\ref{t:endo-koszul}, since the extension algebra $E(A)$ is
  positively graded. A more general argument rests upon 
  Remark \ref{rem:socle-filt}:
  The explicit description of the socle filtration of
  an object of $\dgFilMod$ there and the
  particular form of $\mathcal{K}(A)$ show that 
  \begin{equation*}
    \soc_{i+1} \mathcal{K}(A) = P_0 \oplus \dots \oplus \{i\}P_{-i}.
  \end{equation*}
  Now endomorphisms of $\mathcal{K}(A)$ in $\heart$ respect the socle
  filtration and are therefore upper triangular.
  \remarkend
\end{remark}

\begin{remark}
  \label{rem:koszul-shadow}
  A. Beilinson, V. Ginzburg, and W. Soergel prove an equivalence of triangulated
  categories
  \begin{equation*}
    K:D^b(\gmod(B)) \sira D^b(\gmod(A))
  \end{equation*}
  for $B$ a Koszul ring satisfying some finiteness conditions and
  $A=E(B)$ its Koszul dual ring (see
  \cite[Thm.~1.2.6]{BGS}; we have adapted their statement to our setting, in
  particular we use our definition of the Koszul dual ring using right
  modules). The finiteness conditions are: $B$ is a
  finitely generated $B^0$-module from the right and from the left,
  $A=E(B)$ is right Noetherian,
  and (this condition seems to be missing in \cite{BGS}) 
  the right $B$-module $\leftidx{^*}{B}{}:=\Hom_{\Mod(B^0)}(B,B^0)$ is
  finitely generated); here $D^b(\gmod(B))$ is the 
  bounded derived category of the category of finitely generated
  graded left $B$-modules, and similarly for $D^b(\gmod(A))$.
 
  The standard t-structure on $D^b(\gmod(B))$ with heart $\gmod(B)$
  corresponds (under this equivalence) to some non-standard
  t-structure on $D^b(\gmod(A))$. Hence $\gmod(B)$ is equivalent to the
  heart of this non-standard t-structure.

  Now assume that the Koszul ring $A=E(B)$ satisfies the finiteness
  conditions of Theorem~\ref{t:endo-koszul} and consider
  $\mathcal{A}=(A,d=0)$.
  (For example, one could take the exterior algebra
  $B=\bigwedge V$ of a finite dimensional vector space $V$, and
  $A=SV^*$ the symmetric algebra on the dual space of $V$.) There is a
  natural forgetful functor 
  $D^b(\gmod(B)) \ra D^b(\Mof(B))$ (induced by $\gmod(B)
  \ra \Mof(B)$), and there is a
  functor $D^b(\gmod(A)) \ra \dgPerDer(\mathcal{A})$. From the above
  equivalence between the hearts it is reasonable to expect (as a
  shadow of Koszul duality) an
  equivalence 
  $\Mof(B) \cong \heart$,
  where $\heart$ is the heart of our t-structure on
  $\dgPerDer(\mathcal{A})$. 
  Given our assumptions on $B$ and $A$, this is in fact true:
  Theorem~\ref{t:endo-koszul} yields an equivalence $\heart \sira 
  (\Mofover{B})^\opp$ (note that $E(A)=E(E(B))=B$), and the functor
  $\Hom_{\Mof(B)}(?, \leftidx{^*}{B}{}): \Mof(B) \ra (\Mofover{B})^\opp$ is
  an equivalence, since $\leftidx{^*}{B}{}$
  is an
  injective generator of $\Mof(B)$ and
  $\End_{\Mof(B)}(\leftidx{^*}{B}{})=B$ (use
  $K(\leftidx{^*}{B}{})=A^0$ and some arguments from the proof of Theorem~\ref{t:endo-koszul}). 
  \remarkend
\end{remark}

\bibliographystyle{alpha}
\def\cprime{$'$} \def\cprime{$'$} \def\cprime{$'$} \def\cprime{$'$}

\end{document}